\documentclass[12pt, reqno,oneside]{amsart}

\usepackage{geometry}
\usepackage{graphicx}
\usepackage{color}		
\usepackage{amssymb}
\usepackage{amsthm}
\usepackage{amsmath}
\usepackage{hyperref}
\usepackage{mathrsfs}
\usepackage{mathtools}
\usepackage[mathscr]{eucal}
\usepackage[T1]{fontenc}
\usepackage[toc]{appendix}
\usepackage{dsfont}
\usepackage{indentfirst}
\hypersetup{
    colorlinks=true,
    linkcolor=blue,
    filecolor=magenta,      
    urlcolor=cyan,
    pdftitle={Centered Moments of Weighted One-Level Densities of GL(2) L-Functions},
    pdfpagemode=FullScreen,
    citecolor=blue,
}
\usepackage[style=alphabetic, maxnames=50]{biblatex}
\bibliography{Lfns_PANTHers25_References.bib}

\theoremstyle{plain}
\newtheorem{theorem}{Theorem}[section]
\newtheorem{lemma}[theorem]{Lemma}

\newtheorem{claim}[theorem]{Claim}

\theoremstyle{definition}

\theoremstyle{remark}
\newtheorem{remark}[theorem]{Remark}
\theoremstyle{plain}

\numberwithin{equation}{section}

\allowdisplaybreaks

\let\oldtocsection=\tocsection

\let\oldtocsubsection=\tocsubsection

\let\oldtocsubsubsection=\tocsubsubsection

\renewcommand{\tocsection}[2]{\hspace{0em}\oldtocsection{#1}{#2}}
\renewcommand{\tocsubsection}[2]{\hspace{2em}\oldtocsubsection{#1}{#2}}
\renewcommand{\tocsubsubsection}[2]{\hspace{4.5em}\oldtocsubsubsection{#1}{#2}}
\usepackage{xpatch}
\makeatletter   
\xpatchcmd{\@tocline}
{\hfil\hbox to\@pnumwidth{\@tocpagenum{#7}}\par}
{\ifnum#1<0\hfill\else\dotfill\fi\hbox to\@pnumwidth{\@tocpagenum{#7}}\par}
{}{}
\makeatother

\makeatletter

\title{Centered Moments of Weighted One-Level Densities of $GL(2)$ $L$-Functions}

\author{Lawrence Dillon}
\address{Department of Mathematics, University of Washington, Seattle, WA 98105}
\email{\href{mailto:}{lvid@uw.edu}}

\author{Xiaoyao Huang}
\address{Department of Mathematics, University of Michigan, Ann Arbor, MI 48109}
\email{\href{mailto:}{xyrushac@umich.edu}}

\author{Say-Yeon Kwon}
\address{Department of Mathematics, Princeton University, Princeton, NJ 08544}
\email{\href{mailto:sk9017@princeton.edu}{sk9017@princeton.edu}}

\author{Meiling Laurence}
\address{Department of Mathematics, Yale University, New Haven, CT 06520}
\email{\href{mailto:}{meiling.laurence@yale.edu}}

\author{Steven J. Miller}
\address{Department of Mathematics \& Statistics, Williams College, Williamstown, MA 01267}
\email{\href{mailto:sjm1@williams.edu}{sjm1@williams.edu}}

\author{Vishal Muthuvel}
\address{Department of Mathematics, Columbia University, New York, NY 10027}
\email{\href{mailto:vm2696@columbia.edu}{vm2696@columbia.edu}}

\author{Luke Rowen}
\address{Department of Mathematics, Carleton College, Northfield, MN 55057 }
\email{\href{mailto:rowenl@carleton.edu}{rowenl@carleton.edu}}

\author{Pramana Saldin}
\address{Department of Mathematics, University of Wisconsin, Madison, WI 53706}
\email{\href{mailto:saldin@wisc.edu}{saldin@wisc.edu}}

\author{Steven Zanetti}
\address{Department of Mathematics, University of Michigan, Ann Arbor, MI 48109}
\email{\href{mailto:szanetti@umich.edu}{szanetti@umich.edu}}
\date{\today}
\thanks{This study took place at the SMALL REU at Williams College and was supported by Williams College and the National Science Foundation (Grant DMS2241623). The authors are also grateful for the support of Columbia University, Princeton University, University of Michigan, University of Washington, University of Wisconsin, Williams College, and Yale University. We thank Andrew Knightly for suggesting improvements to the paper, and also our journal referee for detailed and helpful comments.}

\begin{document}
\begingroup

%
%

\begin{abstract}

Katz and Sarnak conjectured that the behavior of zeros near the central point of any family of $L$-functions is well-modeled by the behavior of eigenvalues near $1$ of some classical compact group (either the symplectic, unitary, or even, odd, or full orthogonal group). In 2018, Knightly and Reno proved that the symmetry group can vary depending on how the $L$-functions in the family are weighted. They observed both orthogonal and symplectic symmetry in the one-level densities of families of cuspidal newform $L$-functions for different choices of weights. We observe the same dependence of symmetry on weights in the $n$\textsuperscript{th} centered moments of these one-level densities, for smooth test functions whose Fourier transforms are supported in $\left(-\frac{1}{2n}, \frac{1}{2n}\right)$. To treat the new terms that emerge in our $n$-level calculations when $n>1$, i.e., the cross terms that emerge from $n$-fold products of primes rather than individual primes, we generalize Knightly and Reno's weighted trace formula from primes to arbitrary positive integers. We then perform a delicate analysis of these cross terms to distinguish their contributions to the main and error terms of the $n$\textsuperscript{th} centered moments. The final novelty here is an elementary combinatorial trick that we use to rewrite the main number theoretic terms arising from our analysis, facilitating comparisons with random matrix theory. 

\end{abstract}

\maketitle
\vspace{-1cm}
\tableofcontents

%
%

\section{Introduction}\label{sec:1}

Many studies have been undertaken over the last few decades to make precise the connection between number theory and random matrix theory. Before situating our work in this chronology, we recount the most relevant results and how they prompted the refinement of this connection. 

In 1972, Montgomery \cite{Mon73} formalized the connection between number theory and random matrix theory, proving that, for suitable test functions, the pair correlation of the zeros of the Riemann zeta function $\zeta(s)$ agrees with the pair correlation of the eigenvalues of the Gaussian Unitary Ensemble (GUE). In his seminal work, he furthermore conjectured a correspondence between any local statistic of the zeros of $\zeta(s)$ and the eigenvalues of the GUE. In 1987, Odlyzko \cite{Odl87, Odl01} provided numerical evidence for this conjecture, verifying the correspondence for particular local statistics: the pair correlation and the nearest neighbor spacing distribution. Subsequent works by Hejhal \cite{Hej94} on the triple correlation of $\zeta(s)$ and Rudnick and Sarnak \cite{RS96} on the $n$-level correlation of any automorphic cuspidal $L$-function, all for suitably restricted test functions, provided strong evidence for Montgomery's conjecture in generality. These studies suggested a remarkable universality in number theory: the statistical profile of any $L$-function matches that of one and only one of the numerous random matrix ensembles, the GUE.

At the time, however, there were a few reasons to believe that this conjecture did not capture the connection in full. For one, certain local statistics, such as the $n$-level correlation, are insensitive to the behavior of any finite set of zeros. As there are many important problems in number theory concerning only finite sets of zeros (e.g., the Birch and Swinnerton-Dyer conjecture \cite{BS63, BS65} is concerned only with the low-lying zeros of elliptic curve $L$-functions), this marked a considerable shortcoming of the conjecture on the number theory side. Katz and Sarnak \cite{KS99a, KS99b} confirmed that indeed more attention is needed in this regard. They showed that the $n$-level correlation of eigenvalues is the same coming from the GUE and all five classical compact groups (unitary, symplectic, and orthogonal, split and unsplit by sign). They addressed this possibility of confounding by defining a new statistic called the one-level density, which is not only distinguishable across the classical compact groups but also sensitive to changes in low-lying zeros (near the central point). For any $L$-function $L(s,f)$, denoting its non-trivial zeros by $\rho_f=1/2+i\gamma_f$\footnote{If the Grand Riemann Hypothesis (GRH) is true, then $\gamma$ would be real, and the non-trivial zeros $\rho$ could be ordered accordingly. There are interesting interpretations of this ordering in the context of random matrix theory and nuclear physics, actuating connections to the eigenvalues of classical compact groups and the energy levels of heavy nuclei. With this being said, we do not assume the truth of the GRH in what follows, performing our study in generality.}, they defined its one-level density by 
\begin{equation}\label{eq:1-level-density}
   D(f,\phi) \ := \ \sum_{\rho_f}\phi\left(\frac{\gamma_f\log Q_f}{2\pi}\right), 
\end{equation}
where $\phi$ is an even Schwartz (test) function, and $Q_f$ is the analytic conductor of $f$. The low-lying zeros of $L(s,f)$ have imaginary part approximately $1/\log Q_f$ and the average spacing between zeros with imaginary part $T$ is known to be approximately $1/\log T$. Hence, the scaled zeros, $\frac{\gamma_f\log Q_f}{2\pi}$, have average spacing $1$, meaning there are only finitely many up to any real number $C > 0$. As it is not instructive to statistically survey the finitely many low-lying zeros of a single $L$-function, Katz and Sarnak passed to the study of collections of naturally related (similarly behaved) $L$-functions— ``families'' so to speak. In doing so, they were able to execute averages and take limits, as is customary in analytic number theory, to identify the common behavior underlying the $L$-functions in the family.   

To facilitate our discussion, we establish a standard notation for averages, writing the average of a map $Q$ over a finite collection $\mathcal{C}$ with weight $w$ as 
\begin{equation}\label{eq:arbitrary-average}
\mathcal{A}^w_\mathcal{C}(Q) \ := \ \frac{\sum_{c\in\mathcal{C}}w(c)Q(c)}{\sum_{c\in\mathcal{C}}w(c)}.
\end{equation}
For unweighted averages, i.e., when $w$ is identically $1$ on $\mathcal{C}$, we omit the superscript in \eqref{eq:arbitrary-average}.

In refining Montgomery's conjecture based on their seminal study, Katz and Sarnak formulated the now celebrated density conjecture, stating that the behavior of zeros near $1/2$ in any family of $L$-functions matches the behavior of eigenvalues near $1$ in one of the five classical compact groups. More precisely, consider a family $\mathcal{F}=\bigcup_k \mathcal{F}_k$ of $L$-functions, where each sub-family $\mathcal{F}_k$ is finite. Let the average one-level density of each sub-family be given by
\begin{equation}\label{eq:family-density}
    \mathcal{A}_{\mathcal{F}_k}(D(\cdot,\phi)) \ := \ \frac{\sum_{L(s,f)\in\mathcal{F}_k}D(f,\phi)}{|\mathcal{F}_k|}.
\end{equation}
Then, there exists a symmetry group $G$ among the classical compact groups O, SO(even), SO(odd), Sp, and U such that for any test function $\phi$ with compactly supported Fourier transform $\widehat{\phi}$, 
\begin{equation}\label{eq:density-conjecture}
\lim_{k\to\infty} \mathcal{A}_{\mathcal{F}_k}(D(\cdot,\phi)) \ = \ \int_{-\infty}^\infty\phi(x)W_G(x)dx.
\end{equation}
Here, $W_G(x)$ is the limiting distribution of the one-level density of the eigenvalues coming from the random matrices in $G$ as rank tends to $\infty$. These distributions are given by
\begin{align}
    W_\text{U}(x) \ &= \ 1, \\
    W_\text{Sp}(x) \ &= \ 1-\frac{\sin(2\pi x)}{2\pi x}, \\
    W_\text{SO(even)}(x) \ &= \ 1+\frac{\sin(2\pi x)}{2\pi x}, \\
    W_\text{SO(odd)}(x) \ &= \ \delta_0+1-\frac{\sin(2\pi x)}{2\pi x}, \\
    W_\text{O}(x) \ &= \ 1+\frac{1}{2}\delta_0(x),
\end{align} 
where $\delta_0$ is the Dirac distribution at $0$. When $\widehat{\phi}$ is supported in $(-1,1)$, which will be the case throughout our study, $W_\text{O}(x)$ and $W_\text{Sp}(x)$, and $W_\text{SO(even)}$ and $W_\text{SO(odd)}$ coincide as distributions. This can be realized through Plancherel's theorem \cite[(1.34)]{ILS00}. 

There is now an enormous body of work showing that the one-level densities of various families of $L$-functions (coming from Dirichlet characters, elliptic curves, cuspidal newforms, Maass forms, number fields, etc.) agree with the scaled limit of one of the five classical compact groups; for examples, see \cite{Alp+15,Bar+17,BCL24,DM06,FM15,GK12,ILS00,KR18,RR07,You06}. With this being said, the one-level density is not the only local statistic that provides information about low-lying zeros. Katz and Sarnak defined a higher-order analog for the one-level density called the $n$-level density. For any $L$-function $L(s,f)$, its $n$-level density entails $n$-tuples of its zeros $\rho_f^{(1)},\dots,\rho_f^{(n)}$ and $n$ test functions $\phi_1,\dots,\phi_n$:
\begin{equation}\label{eq:n-level-density}
    D_n(f;\phi_1,\dots,\phi_n) \ := \ \sum_{\substack{\rho_f^{(1)},\dots,\rho_f^{(n)} \\ \gamma_f^{(i)}\neq\pm\gamma_f^{(j)}}}\phi_1\left(\frac{\gamma_f^{(1)}\log Q_f}{2\pi}\right)\cdots\phi_n\left(\frac{\gamma_f^{(n)}\log Q_f}{2\pi}\right).
\end{equation} 
The $n$-level density of a family is usually calculated in terms of its one-level density using the principle of inclusion-exclusion \cite{Rub01, Gao08}. This approach, however, relies on our knowledge of the distribution of the signs of the functional equations in the family, which is not known in general.\footnote{For the family of weight $k$, level $N$ cuspidal newforms that we study in this paper, the distribution of the signs of the functional equations has been formulated by Martin \cite[Theorem 1.2]{Mar23}. Using this, it would be interesting to consider extending our results here from $n$\textsuperscript{th} centered moments of the one-level density to $n$-level densities.} In this view, Hughes and Rudnick \cite{HR02} initiated the study of a more tractable $n$-level statistic: the $n$\textsuperscript{th} centered moment of the one-level density. More precisely, for the general sub-family $\mathcal{F}_k$ of $L$-functions given above, the $n$\textsuperscript{th} centered moment of its one-level density is defined as
\begin{align}\label{eq:centered-moment}
    \mathcal{A}_{\mathcal{F}_k}\left[\left(D(\cdot,\phi)-\mathcal{A}_{\mathcal{F}_k}(D(\cdot,\phi))\right)^n\right] \ := \ \frac{\sum_{L(s,f)\in\mathcal{F}_k}\left(D(f,\phi)-\mathcal{A}_{\mathcal{F}_k}(D(\cdot,\phi))\right)^n}{|\mathcal{F}_k|}.
\end{align}
The $n$\textsuperscript{th} centered moment of the one-level density \eqref{eq:centered-moment} is equivalent to the $n$-level density (\ref{eq:n-level-density}) when $\phi_1=\cdots=\phi_n$, explaining why these statistics encode the same useful arithmetic information, e.g., the order of vanishing at the central point. There are cases \cite{Che+25} in which taking different $\phi_1,\dots,\phi_n$ has proved more productive than taking $\phi_1=\cdots=\phi_n$ in calculating the $n$-level density, cases in which its equivalence to the $n$\textsuperscript{th} centered moment of the one-level density is broken. 

The general approach in calculating these density statistics for families of $L$-functions is to first find an explicit formula converting the sum over zeros in \eqref{eq:1-level-density} to a sum over primes, then an asymptotic (trace) formula for the resulting sum over primes. Finding the explicit formula is usually straightforward, e.g., \cite[Lemma 4.1]{ILS00}, but finding the trace formula usually entails intricate equidistribution laws governing the family. In the specific case of families of automorphic forms, trace formulae naturally involve weights. For example, harmonic weights naturally arise in the Petersson trace formula for families of cuspidal newforms \cite[(2.53)]{ILS00}, and analytic and harmonic weights naturally arise in the Kuznetsov and Selberg trace formulae for families of Maass forms \cite[Remark 1.6]{KL13, GK12}. In this view, many statistical surveys on families of automorphic $L$-functions \cite{Alp+15, GK12, ILS00} maintain these weights and calculate weighted, rather than unweighted, densities of their zeros. In all the above surveys, the weights have proved to be innocuous, in that they do not affect the symmetry type of the low-lying zeros, i.e., the uniform and weighted zeros have the same limiting distribution. However, this is not true in general. 

In 2012, Kowalski, Saha, and Tsimmerman \cite{KST12} demonstrated that more attention is needed in passing from weighted to unweighted densities. They found that the (unweighted) zeros of $GSp(4)$ spinor $L$-functions have orthogonal symmetry; whereas the same zeros, when assigned harmonic weights, exhibit symplectic symmetry. It is reasonable to expect that the distribution of a collection can change based on how the individual elements are weighted. Knightly and Reno \cite{KR18} confirmed this expectation in the commonly studied case of cuspidal newform $L$-functions. They observed both orthogonal and symplectic symmetry in families of holomorphic cuspidal newforms for different choices of weights. 

To be exact, we fix a real, primitive Dirichlet character $\chi$ of modulus $D\geq1$, and a positive integer $r$ relatively prime to $D$. Consider the Gauss sum $\tau(\chi)$ attached to $\chi$ (see \eqref{eq:gauss-sum} for the definition), and a basis $\mathcal{F}_k(N)'$ of Hecke cuspidal newforms of weight $k$ and level $N$. Knightly and Reno considered the following two families:
\begin{itemize}
    \item $\mathcal{F}_1=\bigcup_k \mathcal{F}_k(1)'$ with $k$ ranging over even integers satisfying $\tau(\chi)^2\neq-i^kD$;
    \item $\mathcal{F}_2=\bigcup_{k,N}\mathcal{F}_k(N)'$ with $k>2$ ranging over even integers satisfying $\tau(\chi)^2=-i^kD$, and $N$ ranging over primes not dividing $rD$. We order the sub-families $\mathcal{F}_k(N)'$ so that $k+N$ is increasing. 
\end{itemize}
The conditions on $k$ and $N$ serve to simplify our computations, as we demonstrate in the beginning of Section \ref{sec:4}. To any Hecke cuspidal newform $f(z)=\sum_{n=1}^\infty a_f(n)e^{2 \pi inz}$ in these families, they assigned the weight 
\begin{equation}\label{eq:weights}
  w_{\chi,r}(f) \ := \ \frac{\Lambda\left(\frac{1}{2},f\times\chi\right)|a_f(r)|^2}{\lVert f\rVert^2},
\end{equation}
where $\Lambda\left(s,f\times\chi\right)$ is the completed $L$-function, defined in \eqref{eq:completed-l-function}. These weights are non-negative by Guo's theorem \cite{Guo96}. It is natural to study weights containing central (twisted) $L$-values given the original interest in zeros near the central point \cite{Faz24}, not to mention the Fourier coefficients in the weights also contain information about central $L$-values. With this setup, Knightly and Reno considered the average weighted one-level density of each sub-family $\mathcal{F}_k(N)'$:
\begin{equation}\label{eq:average-density}
\mathcal{A}^{w_{\chi,r}}_{\mathcal{F}_k(N)'}(D(\cdot,\phi)) \ := \ \frac{\sum_{f\in\mathcal{F}_k(N)'}w_{\chi,r}(f)D(f,\phi)}{\sum_{f\in\mathcal{F}_k(N)'}w_{\chi,r}(f)}.
\end{equation}
They proved that for suitably restricted test functions $\phi$, the limiting distribution of the average weighted one-level densities for both $\mathcal{F}_1$ and $\mathcal{F}_2$ varies depending on the triviality of the twisting character in the weight \eqref{eq:weights}:
\begin{align}\label{eq:knightly-reno}\nonumber
\lim_{k\to\infty}\mathcal{A}^{w_{\chi,r}}_{\mathcal{F}_k(1)'}\left(D(\cdot,\phi)\right)
\ &= \ \lim_{k+N\to\infty}\mathcal{A}^{w_{\chi,r}}_{\mathcal{F}_k(N)'}\left(D(\cdot,\phi)\right) \\ & = \  \begin{cases}
    \int_{-\infty}^\infty\phi(x)W_\text{Sp}(x)dx & \text{if }\chi\text{ trivial}, \\ 
    \int_{-\infty}^\infty\phi(x)W_\text{O}(x)dx & \text{if }\chi\text{ non-trivial},
\end{cases}
\end{align}
where the limits are over the aforementioned values of $k$ and $N$. It is instructive to compare \eqref{eq:knightly-reno} to \cite[Theorem 1.1]{ILS00}, which establishes that the limiting distribution of the average unweighted one-level densities for similar families of Hecke cuspidal newforms is orthogonal. In the case of the trivial twisting character, the twisted $L$-function is equal to the original $L$-function; hence the central value of the (twisted) $L$-function is strongly correlated to the zeros near the central point of the original $L$-function (a large central $L$-value generally implies few low-lying zeros). This explains why trivial weighting characters change the limiting distribution from orthogonal to symplectic. On the other hand, in the case of non-trivial twisting characters, the twisted $L$-function (in particular, the twisted central value) is presumably independent of the original $L$-function (in particular, its low-lying zeros). This is to say that these non-trivial characters make for somewhat random weights, independent of the low-lying zeros that they are weighting, so the distribution, in the limit, converges as expected to the orthogonal distribution.

We consider the other arithmetically insightful statistic, the $n$\textsuperscript{th} centered moment of the one-level density, for the same families and weights in an effort to generalize the dependence between weights and symmetry type. To be exact, we study the weighted $n$\textsuperscript{th} centered moment of the average one-level density of each sub-family $\mathcal{F}_k(N)'$:
\begin{align}\label{eq:average-moment}\nonumber
\mathcal{A}^{w_{\chi,r}}_{\mathcal{F}_k(N)'}&\left[\left(D(\cdot,\phi)-\mathcal{A}^{w_{\chi,r}}_{\mathcal{F}_k(N)'}(D(\cdot,\phi))\right)^n\right] \\
&:= \  \frac{\sum_{f\in\mathcal{F}_k(N)'}w_{\chi,r}(f)\left(D(f,\phi)-\mathcal{A}^{w_{\chi,r}}_{\mathcal{F}_k(N)'}(D(\cdot,\phi))\right)^n}{\sum_{f\in\mathcal{F}_k(N)'}w_{\chi,r}(f)}.
\end{align}
In addition to the notation so far, we define for any test function $\phi$, 
\begin{equation}
    \sigma_\phi \ := \ \left(\int_{-\infty}^\infty \widehat\phi(y)^2 |y| \, dy\right)^{1/2}.
\end{equation}
Our main result is that for suitably restricted test functions $\phi$, the weighted $n$\textsuperscript{th} centered moment of the average one-level densities for both $\mathcal{F}_1$ and $\mathcal{F}_2$ converges to the $n$\textsuperscript{th} centered moment of a Gaussian distribution with variance $\sigma_\phi^2$. 
\begin{theorem}\label{thm:main-thm} 
Let $\chi$ be a real, primitive Dirichlet character of modulus $D\geq1$, and let $r$ be a positive integer relatively prime to $D$. For any positive integer $n$ and any test function $\phi$ with $\text{supp}(\widehat{\phi})\subset\left(-\frac{1}{2n},\frac{1}{2n}\right)$,
    \begin{align}\label{eq:main-formula-1}
    \lim_{k\to\infty}\mathcal{A}^{w_{\chi,r}}_{\mathcal{F}_k(1)'}\left[\left(D(\cdot,\phi)-\mathcal{A}^{w_{\chi,r}}_{\mathcal{F}_k(1)'}(D(\cdot,\phi))\right)^n\right] \ = \ 
    &\begin{cases}
    (n-1)!! \ \sigma_{\phi}^n & \text{if } n \text{ even}, \\ 
    0 & \text{if } n \text{ odd},
    \end{cases}\\
    \label{eq:main-formula-2} \lim_{k+N\to\infty}\mathcal{A}^{w_{\chi,r}}_{\mathcal{F}_k(N)'}\left[\left(D(\cdot,\phi)-\mathcal{A}^{w_{\chi,r}}_{\mathcal{F}_k(N)'}(D(\cdot,\phi))\right)^n\right] \ = \ 
    &\begin{cases}
    (n-1)!! \ \sigma_{\phi}^n & \text{if } n \text{ even}, \\ 
    0 & \text{if } n \text{ odd}.
    \end{cases}
    \end{align}
\end{theorem}

\begin{remark}\label{rmk:diff}
It would be interesting to increase the support beyond $(-\frac{1}{n},\frac{1}{n})$ because we can start to distinguish between $W_\text{O}(x)$ and $W_\text{Sp}(x)$ as distributions only beyond this range, as noted in the opening discussion. In the current scope, Theorem \ref{thm:main-thm} does not detect the difference between the orthogonal and symplectic distributions, and hence the dependence between weights (particularly the triviality of the twisting character therein) and symmetry type. 
\end{remark}

\begin{remark}\label{rmk:rmt}
The number theoretic statement of Theorem \ref{thm:main-thm} aligns with random matrix theory: for similarly restricted test functions $\phi$, the $n$\textsuperscript{th} centered moment of the one-level densities for both the orthogonal and symplectic matrix ensembles \cite{HM07, HR07} also converge to the $n$\textsuperscript{th} centered moment of a Gaussian distribution with variance $\sigma_\phi^2$. However, when the support is beyond $[-\frac{1}{n},\frac{1}{n}]$, these random matrix statistics are known to no longer follow Gaussian behavior. Extending the support and realizing the same non-Gaussian behavior on the number theory side, i.e., in the setting of Theorem \ref{thm:main-thm}, would provide striking evidence for the conjectured connection between these two sub-fields.
\end{remark}

\begin{remark}
When $n=1$, we use one of the two main results from the study of Knightly and Reno: \eqref{eq:knightly-reno}. In their other main result \cite[Theorem 1.3]{KR18}, Knightly and Reno considered different weights:
\begin{equation}\label{eq:alternate-weight}
    w_{\chi}'(f)=\frac{\Lambda\left(\frac{1}{2},f\times\chi\right)\Lambda\left(\frac{1}{2},f\right)}{\lVert f\rVert^2}.
\end{equation}
It would be instructive to study the $n$\textsuperscript{th} centered moment of the one-level density with these weights as well, to further demonstrate the dependence between weights and symmetry type. 
\end{remark}

To prove Theorem \ref{thm:main-thm}, we follow the general approach for density calculations outlined above. In Section \ref{sec:2}, we review this relevant mathematical context in more detail and collect some standard results on modular forms.

In Section \ref{sec:3}, we apply the explicit formula for the one-level density of cuspidal newform $L$-functions (Lemma \ref{lemma:explicit-formula}), which leads us from an average of zeros to an average of Hecke eigenvalues over the sub-families $\mathcal{F}_k(N)'$. While Knightly and Reno \cite[Proposition 3.1]{KR18} also considered averages of Hecke eigenvalues over $\mathcal{F}_k(N)'$, they did so only at prime powers because it sufficed for their one-level calculations. The averages that emerge in our $n$-level calculations entail Hecke eigenvalues at arbitrary integers. This complexity is generated by cross terms coming from $n$-fold products of primes rather than individual primes (one-fold products of primes, so to speak). It may be instructive to compare \eqref{eq:2} and \cite[(4.1)]{KR18} in this regard.

In Section \ref{sec:4}, we appeal to a formula given in \cite[Theorem 1.1]{JK15} for a related weighted average of Fourier coefficients over $\mathcal{F}_k(N)'$. Leveraging the relation between Fourier coefficients and Hecke eigenvalues, we derive an asymptotic trace formula (Lemma \ref{lemma:asym-trace-formula}) for the average over $\mathcal{F}_k(N)'$ of Hecke eigenvalues at arbitrary positive integers, not just at prime powers. Several new and interesting number theoretic terms arise in our derivation (see Lemmas \ref{lemma:1} and \ref{lemma:1.2}).

In Section \ref{sec:5}, we analyze these number theoretic terms in several cases, mirroring the proof of \cite[Lemma 3.1]{HM07}. We do this separately for non-trivial and trivial twisting characters, in Subsections \ref{subsec:5.1} and \ref{subsec:5.2} respectively, because the distributions are provably different \eqref{eq:knightly-reno} and the analyses are demonstrably different between these cases. 

In the case of trivial twisting characters, it is not obvious a priori that the number theoretic terms resulting from our analysis match up exactly with the expected random matrix theoretic terms (\ref{eq:main-formula-2}). This apparent mismatch is observed in most, if not all, $n$-level calculations \cite{Rub01, Gao08, HM07, HR02, Sos00}. This is because the random matrix theoretic terms are derived for arbitrary support, cf. \cite{KS99a, KS99b}; whereas the number theoretic terms are derived for restricted support, out of technical necessity. Besides, since these calculations usually use some combinatorial argument (such as the principle of inclusion-exclusion) to express the $n$-level statistic in terms of the corresponding one-level statistic, it is instructive to reverse this combinatorial argument in the end, i.e., unravel the combinatorial expression that emerges after the one-level statistic is calculated. In this vein, Soshnikov \cite{Sos00} developed a combinatorial trick relating to generating series, which has been referenced and adapted in several subsequent $n$-level calculations \cite{HM07, HR02}; the essence of this trick is a deep combinatorial fact called the Hunt-Dyson formula. Novelly, our combinatorial strategy for unraveling the number theoretic terms in Subsection \ref{subsec:5.2} does not rely on this trick. We use only elementary methods to rewrite the resulting combinatorial factors and motivate our otherwise unassuming applications of the binomial theorem. This represents the main challenge in going from Knightly and Reno's one-level calculation, which did not require any combinatorial argument, to our $n$-level calculation. It would be interesting to assess the applicability of our elementary combinatorial argument in other $n$-level calculations, like the ones referenced above. 

\begin{remark}\label{rmk:support}
In view of the many incentives to extending support, e.g., Remarks \ref{rmk:diff} and \ref{rmk:rmt}, it is worth mentioning that the current support restrictions are in place only to bound the contribution of the error term in the asymptotic trace formula (see Lemma \ref{lemma:error-term-analysis}). All other contributions are analyzed unconditionally, for arbitrary support. Extending the support, therefore, entails finding more lower-order terms in the respective trace formula and thereby analyzing a smaller error term. 
\end{remark}
%
%

\section{Notation and Preliminaries}\label{sec:2}

We follow the definitions and notation in \cite{IK04} throughout. We define the $N$\textsuperscript{th} congruence subgroup of $SL_2(\mathbb{Z})$ by

\begin{equation}\label{eq:cong-subgroup}
\Gamma_0(N) \ := \ \left\{ \begin{pmatrix} 
a & b \\ c & d 
\end{pmatrix} \in SL_2(\mathbb{Z}) \ \Bigg| \ c \equiv 0 \pmod{N} \right\}.
\end{equation} 
A holomorphic cusp form of weight $k$ and level $N$ is a function $f$ on the complex upper half-plane $\mathbb{H}$ that transforms ``nicely'' under the action of $\Gamma_0(N)$:
\begin{equation}
f\left( \frac{az + b}{cz + d} \right) = (cz + d)^{k} f(z).
\end{equation}

We denote the space of all holomorphic cusp forms of weight $k$ and level $N$ by $S_k(N)$. This space is equipped with the Petersson inner product, making it Hilbert. The inner product is given by
\begin{equation}\label{eq:Petersson}
\langle f,g\rangle \ := \ \frac{1}{\nu(N)}\int_{\Gamma_0(N) \backslash\mathbb{H}}f(z)\overline{g(z)}\,y^k\,\frac{dx\,dy}{y^2},
\end{equation}
where $\nu(N):=[\mathrm{SL}_2(\mathbb{Z}):\Gamma_0(N)]$. 
Consider the Hecke operators which, for each $n\in \mathbb N$, act on the linear space $S_k(N)$ and are given by
\begin{equation}\label{eq:Hecke}
T_nf(z) \ := \ n^{k-1}\sum_{\substack{ad=n \\(a,N)=1}}\sum_{b=0}^{d-1}d^{-k}f\left(\frac{az+b}{d}\right).  
\end{equation} 
For all $n$ satisfying $(n,N)=1$, the Hecke operators $T_n$ simultaneously diagonalize in $S_k(N)$, so we may speak of an orthogonal basis $\mathcal{F}_k(N)$ of forms which are simultaneous eigenforms of all these Hecke operators (Hecke eigenforms for short). The forms $f\in\mathcal{F}_k(N)$ admit a Fourier expansion
\begin{equation}\label{eq:Fourier}
f(z)=\sum_{n=1}^{\infty}a_f(n)e^{2\pi inz},
\end{equation}
where the coefficients $a_f(n)$ are complex numbers normalized so that $a_f(1)=1$. We emphasize that $\mathcal{F}_k(N)$ is finite; in fact, from \cite[(2.73)]{ILS00},
\begin{equation}
    \left|\mathcal{F}_k(N)\right|\sim\frac{k-1}{12}\phi(N)+O\left((kN)^{5/6}\right),
\end{equation} 
where $\phi$ is Euler's totient function. For any Hecke eigenform $f\in\mathcal{F}_k(N)$, we refer to the eigenvalue of $f$ under $T_n$ as the $n$\textsuperscript{th} Hecke eigenvalue of $f$. It is straightforward from \eqref{eq:Hecke} and \eqref{eq:Fourier} that the Hecke eigenvalues of $f$ are closely related to the Fourier coefficients of $f$: 
\begin{equation}\label{eq:Fourier-Hecke}
    a_f(n) \ = \ \lambda_f(n)n^{(k-1)/2}.
\end{equation}
Furthermore, the Hecke eigenvalues of $f$ possess useful multiplicative properties: for all $m,n\in\mathbb{N}$,
\begin{equation}
\lambda_f(m)\lambda_f(n) \ = \ \sum_{\substack{d\mid(m,n)\\ (d,N)=1}}\lambda_f\left(\frac{mn}{d^2}\right);
\end{equation}
in particular, if $(m,n)=1$, then
\begin{equation}\label{eq:mult-prop}
\lambda_f(m)\lambda_f(n) \ = \ \lambda_f(mn),
\end{equation}
and if $p$ is a prime not dividing $N$ and $\ell \geq 0$ is an integer, then
\begin{align}\label{eq:hecke-comb}
\lambda_f(p)^{\ell} \ &= \ \sum_{\substack{0 \leq r \leq \ell \\ r \equiv \ell (2)}}^\ell c_{\ell,r}\cdot\lambda_f(p^{r}),
\end{align}
where $c_{\ell,r}$ is defined recursively for all $\ell' \geq 1$ by
\begin{align}
\label{eq:hecke-comb-recursive}\nonumber
c_{1,1} \ &:= \ 1, \\
\nonumber
c_{2\ell',2r'} \ &:= \ c_{2\ell'-1,2r'-1}+c_{2\ell'-1,2r'+1} && \text{ for all } 0 < r' < \ell', \\
\nonumber
c_{2\ell',0} \ &:= \ c_{2\ell'-1,1}, \\
\nonumber
c_{2\ell',2\ell'} \ &:= \ c_{2\ell'-1,2\ell'-1}, \\
\nonumber
c_{2\ell'+1,2r'+1} \ &:= c_{2\ell',2r'}+c_{2\ell',2r'+2} && \text{ for all } 0 \leq r' < \ell', \\
c_{2\ell'+1,2\ell'+1} \ &:= c_{2\ell',2\ell'} && \text{ for all } 0 \leq r' < \ell'.
\end{align}
Using Stirling's formula, we have the following crude bound for $c_{\ell,r}$: for all $1 \leq r \leq \ell$,
\begin{equation}\label{eq:hecke-bound}
c_{\ell,r} \ \leq \ 2^{\ell}.
\end{equation}
The multiplicative properties of the Hecke eigenvalues of $f$ motivate the definition of the $L$-function associated to $f$:
\begin{equation}
L(s,f) \ := \ \sum_{n=1}^{\infty}\frac{\lambda_f(n)}{n^s}, \quad \Re(s)>1.
\end{equation}
Interesting variants of this $L$-function can be obtained by twisting it with Dirichlet characters. Formally speaking, for a fixed integer $D\geq1$ with $(D,N)=1$, and a primitive Dirichlet character $\chi$ of modulus $D$, the $L$-function of $f$ twisted by $\chi$ is given by
\begin{equation}
    L(s,f\times\chi) \ := \ \sum_{n=1}^\infty\frac{\chi(n)\lambda_f(n)}{n^s},\quad\Re(s)>1.
\end{equation}
We complete this twisted $L$-function with an appropriate Gamma factor, analytically continuing it to the entire complex plane:
\begin{equation}\label{eq:completed-l-function}
    \Lambda(s,f\times\chi) \ := \ (2\pi)^{-s-\frac{k-1}{2}}\Gamma\left(s+\frac{k-1}{2}\right)L(s,f\times\chi).
\end{equation} 
The completed $L$-function satisfies a functional equation relating $s$ to $1-s$. For example, when $N=1$, the functional equation is
\begin{equation}\label{eq:functional}
    \Lambda(s,f\times\chi) \ = \ \frac{i^k}{D^{2s-1}}\frac{\tau(\chi)^2}{D} \ \Lambda(1-s,f\times\overline{\chi}),
\end{equation} where $\tau(\chi)$ is the Gauss sum attached to $\chi$:
\begin{equation}\label{eq:gauss-sum}
\tau(\chi)\ := \ \sum_{m=1}^D\chi(m)e^{2\pi im/D}.
\end{equation}

With all this said, however, given any form in $S_k(N)$, it is possible to induce a form in $S_k(M)$ for all $M>N$ with $N\mid M$; the induced form is aptly called an ``oldform.'' The forms orthogonal to the space spanned by oldforms are called ``newforms''. This theory, developed by Atkin and Lehner in 1970 \cite{AL70}, is relevant to our study because oldforms of level $M$ are related more naturally to the newforms (of level $N<M$) which induce them than the newforms of level $M$, by way of their analytic conductors. (As mentioned in Section \ref{sec:1}, the analytic conductor of an $L$-function encodes the approximate imaginary part of its low-lying zeros.) In particular, the analytic conductor of any oldform of level $M$ induced from a newform of level $N$ is equal to the analytic conductor of any newform of level $N$. In view of the Katz-Sarnak philosophy of studying families of (naturally related) $L$-functions, we filter out the oldforms in $\mathcal{F}_k(N)$ and direct our interest to the remaining set $\mathcal{F}_k(N)'$ of exclusively newforms, all of which have analytic conductor:
\begin{equation}\label{eq:analytic-conductor}
    Q_{k,N} \ := \ k^2N. 
\end{equation}

We now describe the setup for the rest of our paper. We fix a real, primitive Dirichlet character $\chi$ of modulus $D\geq1$, and a positive integer $r$ relatively prime to $D$. Reiterating Section \ref{sec:1}, we follow Knightly and Reno \cite{KR18} in considering two families, each composed of the finite sub-families $\mathcal{F}_k(N)'$ discussed above:
\begin{itemize}
    \item $\mathcal{F}_1=\bigcup_k \mathcal{F}_k(1)'$ with $k>2$ ranging over even integers satisfying $\tau(\chi)^2\neq-i^kD$;
    \item $\mathcal{F}_2=\bigcup_{k,N}\mathcal{F}_k(N)'$ with $k>2$ ranging over even integers satisfying $\tau(\chi)^2=-i^kD$, equivalently $\chi(-1)=-i^k$, and $N$ ranging over primes not divisible by $rD$. We order the sub-families $\mathcal{F}_k(N)$ so that $k+N$ is increasing.
\end{itemize}
We assign weights to the Hecke newforms $f$ in these families,
\begin{equation}
  w_{\chi,r}(f) \ := \ \frac{\Lambda\left(\frac{1}{2},f\times\chi\right)|a_f(r)|^2}{\lVert f\rVert^2},
\end{equation} 
and examine their influence on the $n$\textsuperscript{th} centered moment of their one-level densities \eqref{eq:average-moment}.

%
%

\section{Explicit Formula}\label{sec:3}

In this section, we restrict our attention to a single sub-family $\mathcal{F}_k(N)'$ of either $\mathcal{F}_1$ or $\mathcal{F}_2$. Given any Hecke eigenform $f\in\mathcal{F}_k(N)'$ and any test function $\phi$ whose Fourier transform $\widehat\phi$ is compactly supported, an explicit formula for the one-level density of $L(s,f)$ is given in \cite[(4.18)]{ILS00}:
\begin{align}
\nonumber
D(f;\phi) \ = \ \widehat{\phi}(0)+\frac{1}{2}\phi(0)+O\left(\frac{\log\log(3N)}{\log Q_{k,N}}\right)&-2\sum_{p\nmid N}\lambda_f(p)\widehat{\phi}\left(\frac{\log p}{\log Q_{k,N}}\right)\frac{\log p}{\sqrt{p}\log Q_{k,N}}\\
\label{eq:second-moment} &-2\sum_{p\nmid N}\lambda_f(p^2)\widehat{\phi}\left(\frac{2\log p}{\log Q_{k,N}}\right)\frac{\log p}{p\log Q_{k,N}}.
\end{align}

Assuming $\text{supp}(\widehat\phi) \subset \left(-\frac{1}{2},\frac{1}{2}\right)$, we know from \cite[88]{ILS00} that the second summation is
\begin{equation}
\sum_{p\nmid N}\lambda_f(p^2)\widehat{\phi}\left(\frac{2\log p}{\log Q_{k,N}}\right)\frac{\log p}{p\log Q_{k,N}} \ \ll \ \frac{\log\log(kN)}{\log Q_{k,N}}.
\end{equation}
Furthermore, Knightly and Reno studied the average weighted one-level density of $\mathcal{F}_k(N)'$ for $\phi$ restricted as above, proving that it depends on the triviality of the twisting character in the weight \cite[165]{KR18}: 
\begin{equation}\label{eq:KR-4.8}
\mathcal{A}^{w_{\chi,r}}_{\mathcal{F}_k(N)'}(D(\cdot,\phi)) \ = \ 
\begin{cases}
    \widehat{\phi}(0)-\frac{1}{2}\phi(0)+O\left(\frac{\log\log(3N)}{\log Q_{k,N}}\right) & \text{if }\chi\text{ trivial},\\
    \widehat{\phi}(0)+\frac{1}{2}\phi(0)+O\left(\frac{\log\log(3N)}{\log Q_{k,N}}\right) & \text{if }\chi\text{ non-trivial}.
\end{cases}
\end{equation}
We are interested in the weighted $n$\textsuperscript{th} centered moment of this average (weighted) one-level density:
\begin{equation}\label{eq:average-moment-1}
\mathcal{A}^{w_{\chi,r}}_{\mathcal{F}_k(N)'}\left[\left(D(\cdot,\phi)-\mathcal{A}^{w_{\chi,r}}_{\mathcal{F}_k(N)'}(D(\cdot,\phi))\right)^n\right].
\end{equation}
Subtracting \eqref{eq:KR-4.8} from \eqref{eq:second-moment}, we determine that \eqref{eq:average-moment-1} equals
\begin{align}\label{eq:1}
\begin{cases}
\mathcal{A}^{w_{\chi,r}}_{\mathcal{F}_k(N)'}\left[\left(\phi(0)-2\sum_{p\nmid N}\lambda_\cdot(p)\widehat{\phi}\left(\frac{\log p}{\log Q_{k,N}}\right)\frac{\log p}{\sqrt{p}\log Q_{k,N}}+O\left(\frac{\log\log(kN)}{\log Q_{k,N}}\right)\right)^n\right] \\ 
\text{ if } \chi \text{ trivial}, \\
\mathcal{A}^{w_{\chi,r}}_{\mathcal{F}_k(N)'}\left[\left(-2\sum_{p\nmid N}\lambda_\cdot(p)\widehat{\phi}\left(\frac{\log p}{\log Q_{k,N}}\right)\frac{\log p}{\sqrt{p}\log Q_{k,N}}+O\left(\frac{\log\log(kN)}{\log Q_{k,N}}\right)\right)^n\right] \\ 
\text{ if } \chi \text{ non-trivial}. \\
\end{cases}
\end{align}
In this view, we attend to averages of the form
\begin{align}\label{eq:1'}
\mathcal{A}^{w_{\chi,r}}_{\mathcal{F}_k(N)'}\left[\left(\sum\nolimits_{p\nmid N}\lambda_\cdot(p) \ \widehat{\phi}\left(\frac{\log p}{\log Q_{k,N}}\right)\frac{\log p}{\sqrt{p}\log Q_{k,N}}\right)^t\right],
\end{align} 
for $0\leq t\leq n$.

\begin{lemma}\label{lemma:explicit-formula}
For any $0\leq t\leq n$, the average in \eqref{eq:1'} is equal to
\begin{align}
\nonumber
\sum_{\substack{1\leq\ell\leq t, \\ n_1+\dots+n_\ell=t}}\frac{t!}{\ell! \ n_1!\cdots n_\ell!}\sum_{\substack{(q_1,\dots,q_\ell), \\ q_j \text{ distinct},\\ q_j\nmid N}}&\left[\left(\prod_{j=1}^\ell\widehat{\phi}\left(\frac{\log q_j}{\log Q_{k,N}}\right)^{n_j}\frac{\log^{n_j} q_j}{q_j^{n_j/2}\log^{n_j} Q_{k,N}}\right)\right. \\ 
\label{eq:4'}
&\left.\times\sum_{\substack{0\leq m_j\leq n_j, \\ m_j\equiv n_j(2)}}\left(\prod_{j=1}^\ell c_{n_j,m_j}\right)\mathcal{A}^{w_{\chi,r}}_{\mathcal{F}_k(N)'}\left(\lambda_\cdot\left(\prod\nolimits_{j=1}^\ell q_j^{m_j}\right)\right)\right],
\end{align}
where all the $c_{n_j,m_j}$ are bounded by $2^t$.
\end{lemma}
\begin{proof}
We expand the product of $t$ sums in \eqref{eq:1'} to get a sum over $t$-tuples:
\begin{align}\label{eq:2}
\sum_{\substack{(p_1,\dots,p_t) \\ p_i\nmid N}}\left[\left(\prod_{i=1}^t\widehat{\phi}\left(\frac{\log p_i}{\log Q_{k,N}}\right)\frac{\log p_i}{\sqrt{p_i}\log Q_{k,N}}\right) \ \mathcal{A}^{w_{\chi,r}}_{\mathcal{F}_k(N)'}\left(\prod\nolimits_{i=1}^t \lambda_\cdot(p_i)\right)\right].
\end{align}
To appeal to the multiplicative properties of Hecke eigenvalues \eqref{eq:mult-prop}, we consider the prime factorization: $\prod_{i=1}^t p_i=\prod_{j=1}^\ell q_j^{n_j}$. However, to change the index of summation from $p_i$ to $q_j$, we require the combinatorial factor indicating the number of $t$-tuples $(p_1,\ldots,p_t)$ with a given prime factorization $\prod_{j=1}^\ell q_j^{n_j}$:
\begin{equation}\label{eq:comb-factor}
\frac{1}{\ell!}\binom{t}{n_1}\binom{t-n_1}{n_2}\cdots\binom{n_\ell}{n_\ell} \ = \ \frac{t!}{\ell! \ n_1!\cdots n_\ell!}.
\end{equation}
With this, we make the desired change of index in \eqref{eq:2} to get
\begin{align}\nonumber
\sum_{\substack{1\leq\ell\leq t, \\ n_1+\dots+n_\ell=t}}\frac{t!}{\ell! \ n_1!\cdots n_\ell!}\sum_{\substack{(q_1,\dots,q_\ell)\\ q_j\nmid N}}&\left[\left(\prod_{j=1}^\ell\widehat{\phi}\left(\frac{\log q_j}{\log Q_{k,N}}\right)^{n_j}\frac{\log^{n_j} q_j}{q_j^{n_j/2}\log^{n_j} Q_{k,N}}\right)\right. \\
\label{eq:3}&\left.\times\mathcal{A}^{w_{\chi,r}}_{\mathcal{F}_k(N)'}\left(\prod\nolimits_{j=1}^\ell \lambda_\cdot(q_j)^{n_j}\right)\right].
\end{align}
As in \eqref{eq:hecke-comb}, we write $\lambda_\cdot(q_j)^{n_j}$ as a linear combination of Hecke eigenvalues at the powers of $q_j$:
\begin{equation}
\lambda_\cdot(q_j)^{n_j} \ = \ \sum_{\substack{0\leq m_j\leq n_j\\ m_j\equiv n_j(2)}}c_{n_j,m_j} \ \lambda_\cdot(q_j^{m_j}),
\end{equation} 
where each $c_{n_j,m_j}$ is bounded by $2^{n_j} \leq 2^t$ (see \eqref{eq:hecke-bound}). Making this substitution in \eqref{eq:3} gives \eqref{eq:4'}.
\end{proof}

Lemma \ref{lemma:explicit-formula} leads our study to weighted averages of Hecke eigenvalues over $\mathcal{F}_k(N)'$.

%
%

\section{Weighted Trace Formula}\label{sec:4}

In this section, we derive an asymptotic trace formula for the weighted average of the $m$\textsuperscript{th} Hecke eigenvalue over $\mathcal{F}_k(N)'$ when $(m,D)=1$: 
\begin{equation}\label{eq:average-Hecke}
\mathcal{A}^{w_{\chi,r}}_{\mathcal{F}_k(N)'}(\lambda_\cdot(m)) \ := \ \frac{\sum_{f\in\mathcal{F}_k(N)'}w_{\chi,r}(f)\lambda_f(m)}{\sum_{f\in\mathcal{F}_k(N)'}w_{\chi,r}(f)}.
\end{equation}

Trace formulas are more readily attainable for $\mathcal{F}_k(N)$ than $\mathcal{F}_k(N)'$. This is because the equidistribution laws that govern Hecke eigenvalues---and thereby influence trace formulas---are better understood for the space of all forms $S_k(N)$ rather than the space of newforms alone. Studying families of newforms usually entails extensive bookkeeping and adapting trace formulas from the ambient space to the space of newforms, e.g., \cite[Propositions 2.1-2.8]{ILS00}. Fortunately, we do not have to do any bookkeeping in our study, owing to the strategic conditioning on $k$ and $N$ in our definitions of $\mathcal{F}_1$ and $\mathcal{F}_2$.

For level $1$, it is straightforward that $\mathcal{F}_k(1)=\mathcal{F}_k(1)'$ because there are no levels lower than $1$, and hence no forms to induce from. 

For prime level $N>1$, as Knightly and Reno explained in \cite[Section 4]{KR18}, the condition involving $\tau(\chi)$ in the definition of $\mathcal{F}_2$ licenses us to bypass bookkeeping. Every form $f\in\mathcal{F}_k(N)$ is either a newform of level $N$ or an oldform of level $N$ induced from a newform $g$ of level $1$. In the latter case, since $\tau(\chi)^2=-i^k D$, the functional equation \eqref{eq:functional} for $\Lambda(s,g\times\chi)$ forces $w_{\chi,r}(g)$ to vanish. For the induced form $g_N(z):=g(Nz)$ too, the weight vanishes:
\begin{align}
    \Lambda\left(\frac{1}{2},g_N\times\chi\right) \ = \ \frac{\chi(N)}{N^{k/2}}\Lambda\left(\frac{1}{2},g\times\chi\right) \ = \ 0 
    \implies w_{\chi,r}(g_N)=0.
\end{align}
It follows that the weight vanishes for all forms in the span of $g_N$, $f$ in particular. This is to say that $w_{\chi,r}(f)$ vanishes for all oldforms $f\in\mathcal{F}_k(N)$; hence any $w_{\chi,r}$-weighed average over $\mathcal{F}_k(N)$ reduces to a $w_{\chi,r}$-weighted average over $\mathcal{F}_k(N)'$, and we can refer to these interchangeably.

In the interest of finding a weighted trace formula for (\ref{eq:average-Hecke}), we call attention to a related formula in \cite[Theorem 1.1]{JK15}. For all complex numbers $s=\sigma+it$ in the strip $1-\frac{k-1}{2}<\sigma<\frac{k-1}{2}$, Jackson and Knightly derived a formula for the sum of Fourier coefficients over $\mathcal{F}_k(N)$ weighted by
\begin{equation}\label{eq:JK-weights}
w_{\chi,r}'(f)\ := \ \frac{\Lambda\left(s,f\times\chi\right)\overline{a_f(r)}}{\lVert f\rVert^2},\quad f\in\mathcal{F}_k(N).
\end{equation}
When $k>2$ and $N=1$, for all $m\geq1$ satisfying $(m,D)=1$, they get
\begin{align}\nonumber
\sum_{f\in\mathcal{F}_k(1)}w_{\chi,r}'(f)a_f(m)& \\
\nonumber 
= \ & \ \frac{2^{k-1}(2 \pi rm)^{\frac{k-1}{2}-s}}{(k-2)!}\Gamma\left(s+\frac{k-1}{2}\right)\sum_{d\mid(r,m)}d^{2s}\chi\left(\frac{rm}{d^2}\right)\\
\nonumber
+ \ & \frac{2^{k-1}(2\pi rm)^{\frac{k-3}{2}+s}}{(k-2)!}\Gamma\left(\frac{k+1}{2}-s\right)\frac{i^k}{D^{2s-1}}\frac{\tau(\chi)^2}{D}\sum_{d\mid(r,m)}d^{-2s+2}\overline{\chi\left(\frac{rm}{d^2}\right)}\\
\label{eq:KR-trace-1.1'}
+ \ & O\left(\gcd(r,m)\frac{(4\pi rm)^{k-1}D^{\frac{k}{2}-\sigma}\varphi(D)}{N^{\sigma+\frac{k-1}{2}}(k-2)!}\right);
\end{align}
likewise, when $k>2$ and $N>1$, for all $m\geq1$ with $(m,D)=1$, they get
\begin{align}\nonumber
\sum_{f\in\mathcal{F}_k(N)}w_{\chi,r}'(f)a_f(m) & \\
\nonumber 
= \ & \ \frac{2^{k-1}(2 \pi rm)^{\frac{k-1}{2}-s}}{(k-2)!}\Gamma\left(s+\frac{k-1}{2}\right)\sum_{d\mid(r,m)}d^{2s}\chi\left(\frac{rm}{d^2}\right)\\
\label{eq:KR-trace-2.1'}
+ \ & O\left(\gcd(r,m)\frac{(4\pi rm)^{k-1}D^{\frac{k}{2}-\sigma}\varphi(D)}{N^{\sigma+\frac{k-1}{2}}(k-2)!}\right).
\end{align}
The implied constants are provided explicitly in \cite[(1.5)]{JK15} and depend only on $s$. As in \cite[160]{KR18}, we have adjusted these equations for the fact that \cite{JK15} normalizes the completed $L$-function $\Lambda(s,f\times\chi)$ to have central point $\frac{k}{2}$ rather than $\frac{1}{2}$; that is, we apply the transformation $s \mapsto s+\frac{k-1}{2}$ to \cite[(1.4)]{JK15} to satisfy our normalization.

In our application, we take $\chi$ to be a real character, so $\chi = \overline{\chi}$ and $\chi^2 = 1$; and $s = 1/2$. We also allow the implied constants to depend on $r$ and $D$ since they are fixed. Specializing \eqref{eq:KR-trace-1.1'} and \eqref{eq:KR-trace-2.1'} in this regard, we have that for all $k>2$ and $m\geq1$ with $(m,D)=1$,
\begin{align}\nonumber
\sum_{f\in\mathcal{F}_k(1)}w_{\chi,r}'(f)a_f(m)& \\
\nonumber 
= \ & \frac{2^{k-1}(2\pi)^{\frac{k}{2}-1}\Gamma\left(\frac{k}{2}\right)}{(k-2)!}\left(1+\frac{i^k\tau(\chi)^2}{D}\right)(rm)^{\frac{k}{2}-1}\chi(rm) \ \sigma_1((r,m))\\
\label{eq:KR-trace-1.1}
+ \ & O\left(m^{k-1}\frac{(4\pi r)^{k}D^{\frac{k}{2}}}{N^{\frac{k}{2}}(k-2)!}\right); \\
\nonumber
\sum_{f\in\mathcal{F}_k(N)}w_{\chi,r}'(f)a_f(m)& \\
\nonumber 
= \ & \frac{2^{k-1}(2\pi)^{\frac{k}{2}-1}\Gamma\left(\frac{k}{2}\right)}{(k-2)!}(rm)^{\frac{k}{2}-1}\chi(rm) \ \sigma_1((r,m))\\
\label{eq:KR-trace-2.1}
+ \ & O\left(m^{k-1}\frac{(4\pi r)^{k}D^{\frac{k}{2}}}{N^{\frac{k}{2}}(k-2)!}\right).
\end{align}
In the following three lemmas, we adapt these asymptotic trace formulas for the weighted average of our interest, i.e., with weights $w_{\chi,r}$ rather than $w_{\chi,r}'$: \eqref{eq:average-Hecke}. 

\begin{lemma}\label{lemma:1}
Suppose $k>2$. Let $\sigma_1$ be the divisor sum function, and $S_r(m) := \sum_{d \mid (r,m)}d \cdot \sigma_1\left(\left(r,\frac{rm}{d^2}\right)\right)$. Then, for any integer $m\geq1$ with $(m,D) = 1$, the sum of the $m$\textsuperscript{th} Hecke eigenvalue over $\mathcal{F}_k(1)$ is equal to
\begin{align}\nonumber
\sum_{f\in\mathcal{F}_k(1)}w_{\chi,r}(f)\lambda_f(m) & \\
\nonumber
= \ & \frac{2^{k-1}(2\pi r^2)^{\frac{k}{2}-1}\Gamma\left(\frac{k}{2}\right)}{(k-2)!}\left(1+i^k\frac{\tau(\chi)^2}{D}\right) m^{-\frac{1}{2}} \chi(m)  S_r(m)\\
\label{eq:sum-1}
+ \ & O\left(m^{\frac{k}{2}-1}\frac{(4\pi r)^{k}D^{\frac{k}{2}}}{N^{\frac{k}{2}}(k-2)!}\right),
\end{align}
where the implied constant depends only on $r$ and $D$. 
\end{lemma}
\begin{proof}
For all $f$, we have that $w_{\chi,r}(f)=w_{\chi,r}'(f)a_f(r)$. We use this relation between weights and the relation \eqref{eq:Fourier-Hecke} between Fourier coefficients and Hecke eigenvalues to pass between the sum of interest \eqref{eq:sum-1} and the sum already analyzed \eqref{eq:KR-trace-1.1}:
\begin{align}
\sum_{f\in\mathcal{F}_k(1)}w_{\chi,r}(f)\lambda_f(m) \ 
&= \ \sum_{f\in\mathcal{F}_k(1)}w_{\chi,r}'(f)r^{\frac{k-1}{2}}\lambda_f(r)\lambda_f(m).
\end{align}
Appealing to the multiplicative properties (\ref{eq:mult-prop}) of Hecke eigenvalues,
\begin{align}
\sum_{f\in\mathcal{F}_k(1)}w_{\chi,r}(f)\lambda_f(m) \ 
&= \ \sum_{f\in\mathcal{F}_k(1)}w_{\chi,r}'(f)r^{\frac{k-1}{2}}\sum_{d\mid(r,m)}\lambda_f\left(\frac{rm}{d^2}\right) \\
&= \ \sum_{f\in\mathcal{F}_k(1)}w_{\chi,r}'(f)r^{\frac{k-1}{2}}\sum_{d\mid(r,m)}\left(\frac{rm}{d^2}\right)^{-\frac{k-1}{2}}a_f\left(\frac{rm}{d^2}\right).
\end{align}
Changing the order of these finite summations,
\begin{align}
\label{eq:recent}
\sum_{f\in\mathcal{F}_k(1)'}w_{\chi,r}(f)\lambda_f(m) \
&= \ \sum_{d\mid(r,m)} \left(\frac{m}{d^2}\right)^{-\frac{k-1}{2}}\sum_{f\in\mathcal{F}_k(1)}w_{\chi,r}'(f)a_f\left(\frac{rm}{d^2}\right).
\end{align}
Appealing to \eqref{eq:KR-trace-1.1} and recalling that $\chi$ is real, i.e., $\chi^2=1$, we find that (\ref{eq:recent}) equals
\begin{align}
\nonumber
&\ \frac{2^{k-1}(2\pi)^{\frac{k}{2}-1}\Gamma\left(\frac{k}{2}\right)}{(k-2)!}\left(1+i^k\frac{\tau(\chi)^2}{D}\right)\sum_{d\mid(r,m)}\left(\frac{m}{d^2}\right)^{-\frac{k-1}{2}}\left(\frac{r^2m}{d^2}\right)^{\frac{k}{2}-1}\chi\left(\frac{r^2m}{d^2}\right)\sigma_1\left(\left(r,\frac{rm}{d^2}\right)\right)\\
\nonumber
&+O\left(\frac{(4\pi r)^{k}D^{\frac{k}{2}}}{N^{\frac{k}{2}}(k-2)!}\sum_{d\mid(r,m)}\left(\frac{m}{d^2}\right)^{-\frac{k}{2}}\left(\frac{rm}{d^2}\right)^{k-1}\right) \\
\nonumber
= \ & \ \frac{2^{k-1}(2\pi r^2)^{\frac{k}{2}-1}\Gamma\left(\frac{k}{2}\right)}{(k-2)!}\left(1+i^k\frac{\tau(\chi)^2}{D}\right)m^{-\frac{1}{2}}\chi(m)\sum_{d\mid(r,m)}d \cdot \sigma_1\left(\left(r,\frac{rm}{d^2}\right)\right)\\
\label{eq:final-asymp}
&+O\left(m^{\frac{k}{2}-1}\frac{(4\pi r^2)^{k}D^{\frac{k}{2}}}{N^{\frac{k}{2}}(k-2)!}\sum_{d\mid(r,m)}d^{-k+1}\right).
\end{align}
Noting that 
\begin{equation}
\sum_{d \mid (r,m)} d^{-k+1} \leq \sum_{d \mid r} d^{-k+1} \leq \sum_{d=1}^r d^{-k+1} \ll_r r^{-k},
\end{equation}
we see \eqref{eq:final-asymp} is equal to
\begin{align}\nonumber
&\frac{2^{k-1}(2\pi r^2)^{\frac{k}{2}-1}\Gamma\left(\frac{k}{2}\right)}{(k-2)!}\left(1+i^k\frac{\tau(\chi)^2}{D}\right)m^{-\frac{1}{2}}\chi(m) S_r(m)\\
&+O\left(m^{\frac{k}{2}-1}\frac{(4\pi r)^{k-1}D^{\frac{k-1}{2}}}{N^{\frac{k}{2}}(k-2)!}\right).
\end{align}
which is precisely the claimed formula in the statement of the lemma. 
\end{proof}

\begin{lemma}\label{lemma:1.2}
Suppose $k>2$ and $N>1$. Let $\sigma_1$ be the divisor sum function, and $S_r(m) := \sum_{d \mid (r,m)}d \cdot \sigma_1\left(\left(r,\frac{rm}{d^2}\right)\right)$. Then, for any integer $m\geq1$ with $(m,D) = 1$, the sum of the $m$\textsuperscript{th} Hecke eigenvalue over $\mathcal{F}_k(N)$ is equal to
\begin{align}\nonumber
\sum_{f\in\mathcal{F}_k(1)}w_{\chi,r}(f)\lambda_f(m) & \\
\nonumber
= \ & \frac{2^{k-1}(2\pi r^2)^{\frac{k}{2}-1}\Gamma\left(\frac{k}{2}\right)}{(k-2)!} m^{-\frac{1}{2}} \chi(m) S_r(m)\\
\label{eq:main-2}
+ \ & O\left(m^{\frac{k-1}{2}}\frac{(4\pi r)^{k}D^{\frac{k}{2}}}{N^{\frac{k}{2}}(k-2)!}\right),
\end{align}
where the implied constant depends only on $r$ and $D$. 
\end{lemma}
\begin{proof}
This proof proceeds similarly to the proof of Lemma \ref{lemma:1}, leveraging the relation between the weights in question, the relation \eqref{eq:Fourier-Hecke} between Fourier coefficients and Hecke eigenvalues, and the multiplicative properties \eqref{eq:mult-prop} of Hecke eigenvalues to adapt \eqref{eq:KR-trace-2.1} to get \eqref{eq:main-2}.
\end{proof}

\begin{lemma}\label{lemma:asym-trace-formula}
Suppose $k>2$ and $N\geq1$. Let $\sigma_1$ be the divisor sum function, and $S_r(m) := \sum_{d \mid (r,m)}d \cdot \sigma_1\left(\left(r,\frac{rm}{d^2}\right)\right)$. Let $m=\prod_{j=1}^\ell q_j^{m_j}$, where $q_1,\ldots,q_j$ are the distinct primes dividing $m$, and $(m,DN) = 1$. Then, the average \eqref{eq:average-Hecke} of the $m$\textsuperscript{th} Hecke eigenvalue over $\mathcal{F}_k(N)$ is equal to
\begin{align}\label{eq:main}
\mathcal{A}^{w_{\chi,r}}_{\mathcal{F}_k(N)}\left(\lambda_\cdot\left(\prod\nolimits_{j=1}^\ell q_j^{m_j}\right)\right)\ &= \ \left(\prod\nolimits_{j=1}^\ell q_j^{m_j}\right)^{-\frac{1}{2}}\left(\prod\nolimits_{j=1}^\ell\chi(q_j)^{m_j}\right)S_r\left(\prod\nolimits_{j=1}^\ell q_j^{m_j}\right)\\
\label{eq:error}
& + \ O\left(\frac{\left(\prod\nolimits_{j=1}^\ell q_j^{m_j}\right)^{\frac{k-1}{2}}W^k}{N^{\frac{k-1}{2}}k^{\frac{k}{2}-1}}\right),
\end{align}
where the implied constant depends only on $r$ and $D$, and $W$ depends only on $D$. 
\end{lemma}

\begin{proof}
First, we prove this statement when $k>2$ and $N>1$. We would like to compute 
\begin{equation}\label{eq:av}
\mathcal{A}^{w_{\chi,r}}_{\mathcal{F}_k(N)}(\lambda_\cdot(m)) \ = \ \frac{\sum_{f\in\mathcal{F}_k(N)'}w_{\chi,r}(f)\lambda_f(m)}{\sum_{f\in\mathcal{F}_k(N)'}w_{\chi,r}(f)}.
\end{equation}
By Lemma \ref{lemma:1.2} (with $m=1$), the denominator equals
\begin{align}
M_D+E_D \ := \ \frac{2^{k-1}(2\pi r^2)^{\frac{k}{2}-1}\Gamma\left(\frac{k}{2}\right)}{(k-2)!}+O\left(\frac{(4\pi r)^k D^{\frac{k}{2}}}{N^{\frac{k}{2}}(k-2)!}\right),  
\end{align}
where $M_D$ and $E_D$ denote the main and error terms, respectively. Following this convention, the numerator equals
\begin{align}
M_N+E_N \ := \ \frac{2^{k-1}(2\pi r^2)^{\frac{k}{2}-1}\Gamma\left(\frac{k}{2}\right)}{(k-2)!} m^{-\frac{1}{2}}\chi(m) S_r(m)+ \ O\left(m^{\frac{k-1}{2}}\frac{(4\pi r)^kD^{\frac{k}{2}}}{N^{\frac{k}{2}}(k-2)!}\right).
\end{align}
Rewriting our target expression \eqref{eq:av} in terms of $M_N,E_N,M_D,$ and $E_D$,
\begin{align}
\nonumber
\mathcal{A}^{w_{\chi,r}}_{\mathcal{F}_k(N)'}(\lambda_\cdot(m)) \ &= \ \frac{M_N+E_N}{M_D+E_D} \\
\label{eq:target}
&= \ \frac{M_N}{M_D}+\frac{E_N-\frac{M_N}{M_D}E_D}{M_D+E_D}.
\end{align}
It is straightforward that 
\begin{align}
\nonumber
\frac{M_N}{M_D} \ &= m^{-\frac{1}{2}}\chi(m) \ S_r(m) \\
&= \ \left(\prod\nolimits_{j=1}^\ell q_j^{m_j}\right)^{-\frac{1}{2}}\left(\prod\nolimits_{j=1}^\ell\chi(q_j)^{m_j}\right)S_r\left(\prod\nolimits_{j=1}^\ell q_j^{m_j}\right),
\end{align} 
matching the main term \eqref{eq:main} in the lemma. We finish by demonstrating that the second term in \eqref{eq:target} has the desired rate of decay from \eqref{eq:error}. For clarity of the analysis, we let \begin{equation}
C\ := \ \frac{(4\pi r)^kD^{\frac{k}{2}}}{N^{\frac{k}{2}}(k-2)!},    
\end{equation} 
under which,
\begin{align}
\nonumber
\frac{E_N-\frac{M_N}{M_D}E_D}{M_D+E_D} \ &\ll \ \frac{\left(m^{\frac{k-1}{2}}-m^{-\frac{1}{2}}\chi(m)S_r(m)\right)C}{M_D+E_D} \\
\nonumber
&\ll_r \ \frac{m^{\frac{k-1}{2}}C}{M_D+E_D}\\
\label{eq:final-4}
&= \ m^{\frac{k-1}{2}}\frac{C/M_D}{1+E_D/M_D}.
\end{align}
Relating $C$ and $M_D$ to the variables $C_0$ and $F_0$ (specialized for $s=\frac{1}{2}$) defined in \cite[160-161]{KR18}, we have
\begin{equation}
M_D \ = \ \frac{r^{\frac{k}{2}-1}}{\chi(r)}F_0, \quad C \ll_{r,D} C_0
\end{equation}
In this regards, the argument 
in \cite[161]{KR18}, which originates from \cite[Section 9]{JK15}, carries through in our analysis:
\begin{align}\label{eq:semifinal-4}
\frac{E_D}{M_D} \ \ll_{r,D} \ \frac{C}{M_D} \ \ll_{r,D} \frac{C_0}{r^{\frac{k}{2}}F_0} \ \ll_{r,D} \ \frac{(4\pi De)^k}{N^{\frac{k-1}{2}}k^{\frac{k}{2}-1}}.
\end{align}
Taking $W = 4 \pi D e$, it follows that \eqref{eq:final-4} has the decay claimed in the statement of the lemma.

The proof for $N=1$ is similar; the extra factor of $\left(1+\frac{i^k\tau(\chi)^2}{D}\right)$ in the main term (\ref{eq:KR-trace-1.1}) will cancel coming from both the numerator and the denominator. 
\end{proof}

%
%

\section{Combinatorial Analysis}\label{sec:5}

In this section, we present our proof of Theorem \ref{thm:main-thm}, consolidating Lemmas \ref{lemma:explicit-formula} and \ref{lemma:asym-trace-formula}. 

We are interested in
\begin{align}\label{eq:3'}
\mathcal{A}^{w_{\chi,r}}_{\mathcal{F}_k(N)'}\left[\left(\sum\nolimits_{p\nmid N}\lambda_\cdot(p) \ \widehat{\phi}\left(\frac{\log p}{\log Q_{k,N}}\right)\frac{\log p}{\sqrt{p}\log Q_{k,N}}\right)^t\right],
\end{align} 
which, by Lemma \ref{lemma:explicit-formula}, is known to equal
\begin{align}
\nonumber
\sum_{\substack{1\leq\ell\leq t, \\ n_1+\dots+n_\ell=t}}\frac{t!}{\ell! \ n_1!\cdots n_\ell!}\sum_{\substack{(q_1,\dots,q_\ell), \\ q_j \text{ distinct}, \\ q_j\nmid N}}&\left[\left(\prod_{j=1}^\ell\widehat{\phi}\left(\frac{\log q_j}{\log Q_{k,N}}\right)^{n_j}\frac{\log^{n_j} q_j}{q_j^{n_j/2}\log^{n_j} Q_{k,N}}\right)\right. \\ 
\label{eq:4}
&\left.\times\sum_{\substack{0\leq m_j\leq n_j, \\ m_j\equiv n_j(2)}}\left(\prod_{j=1}^\ell c_{n_j,m_j}\right)\mathcal{A}^{w_{\chi,r}}_{\mathcal{F}_k(N)'}\left(\lambda_\cdot\left(\prod\nolimits_{j=1}^\ell q_j^{m_j}\right)\right)\right].
\end{align}

Applying the weighted trace formula from Lemma \ref{lemma:asym-trace-formula} to expand \eqref{eq:4}, we find that the contribution of the main term \eqref{eq:main} to \eqref{eq:4} is
\begin{align}\nonumber
\sum_{\ell=1}^t\sum_{n_1+\dots+n_\ell=t}&\sum_{\substack{0\leq m_j\leq n_j\\ m_j\equiv n_j(2) \\ \forall1\leq j\leq\ell}}\frac{t!}{\ell! \ n_1!\cdots n_\ell!}\left(\prod_{j=1}^\ell c_{n_j,m_j}\right)\\
\label{eq:main-final}
&\times\sum_{\substack{(q_1,\dots,q_\ell), \\ q_j \text{ distinct}, \\ q_j\nmid N}}\left(\prod_{j=1}^\ell\widehat{\phi}\left(\frac{\log q_j}{\log Q_{k,N}}\right)^{n_j}\frac{\chi(q_j)^{m_j}\log^{n_j} q_j}{q_j^{(n_j+m_j)/2}\log^{n_j} Q_{k,N}} \ S_r\left(q_j^{m_j}\right)\right);
\end{align}
and the contribution of the error term \eqref{eq:error} to \eqref{eq:4} is 
\begin{align}\label{eq:err-term-contribution}
\ll_{r,t,D}\frac{W^k}{N^{\frac{k-1}{2}}k^{\frac{k}{2}-1}}\sum_{\substack{(q_1,\dots,q_\ell), \\ q_j \text{ distinct},  \\ q_j\nmid N}}\left(\prod_{j=1}^\ell\widehat{\phi}\left(\frac{\log q_j}{\log Q_{k,N}}\right)^{n_j}\frac{\log^{n_j} q_j}{q_j^{n_j/2}\log^{n_j} Q_{k,N}} q_j^{\frac{m_j(k-1)}{2}}\right).
\end{align}
We highlight that \eqref{eq:main-final} uses that $S_r$, being a convolution of multiplicative functions, is multiplicative. Also, note that we absorbed the combinatorial sum in \eqref{eq:4} into the implied constant in \eqref{eq:err-term-contribution} as it is bounded by $2^t$. 

We start by bounding the error term contribution \eqref{eq:err-term-contribution}. As explained in Remark \ref{rmk:support}, this step poses the only obstruction to extending support.

\begin{lemma}\label{lemma:error-term-analysis}
For any test function $\phi$ with $\text{supp}(\widehat{\phi})\subset\left(-\frac{1}{2n},\frac{1}{2n}\right)$, the error term contribution \eqref{eq:err-term-contribution} vanishes when either $k\to\infty$ or $k+N\to\infty$.
\end{lemma}
\begin{proof}
By the principle of inclusion-exclusion, we know that 
\begin{align}
\nonumber
&\frac{W^k}{N^{\frac{k-1}{2}}k^{\frac{k}{2}-1}}\sum_{\substack{(q_1,\dots,q_\ell) \\ q_j\nmid N}}\left(\prod_{j=1}^\ell\widehat{\phi}\left(\frac{\log q_j}{\log Q_{k,N}}\right)^{n_j}\frac{\log^{n_j} q_j}{q_j^{n_j/2}\log^{n_j} Q_{k,N}} q_j^{\frac{m_j(k-1)}{2}}\right) \\
\label{eq:err}
\ll & \ \frac{W^k}{N^{\frac{k-1}{2}}k^{\frac{k}{2}-1}}\prod_{j=1}^\ell\left(\sum_{q\nmid N}\widehat{\phi}\left(\frac{\log q}{\log Q_{k,N}}\right)^{n_j}\frac{\log^{n_j} q}{q^{n_j/2}\log^{n_j} Q_{k,N}} q_j^{\frac{m_j(k-1)}{2}}\right).
\end{align}
Besides, if $\text{supp}(\widehat{\phi})\subset(-\alpha,\alpha)$, then \eqref{eq:err} is bounded by 
\begin{align}
\nonumber
\ll & \ \frac{W^k}{N^{\frac{k-1}{2}}k^{\frac{k}{2}-1}}\prod_{j=1}^\ell\left(\sum_{q\ll Q_{k,n}^\alpha}q^{\frac{m_j(k-1)}{2}-\frac{n_j}{2}}\right) \\
\nonumber
\ll & \ \frac{W^k}{N^{\frac{k-1}{2}}k^{\frac{k}{2}-1}}\prod_{j=1}^\ell\left(\sum_{q\ll Q_{k,n}^\alpha}q^{\frac{n_jk}{2}-n_j}\right) \\
\nonumber
\ll & \ \frac{W^k}{N^{\frac{k-1}{2}}k^{\frac{k}{2}-1}}\prod_{j=1}^\ell Q_{k,N}^{\frac{n_jk\alpha}{2}-n_j\alpha+1}\\
\label{eq:finally}
\ll & \ \frac{W^k}{N^{\frac{k-1}{2}}k^{\frac{k}{2}-1}}Q_{k,N}^{\frac{nk\alpha}{2}-n\alpha+n},
\end{align}
recalling that $m_j\leq n_j$, $\ell\leq t$, and $\sum_{j=1}^\ell n_j=t$.
Since $Q_{k,N}=k^2N$, if $\alpha<\frac{1}{2n}$, then \eqref{eq:finally} vanishes as $k\to\infty$ or $k+N\to\infty$.
\end{proof}

Now, we turn to analyze the main term contribution \eqref{eq:main-final}. Applying the principle of inclusion-exclusion to \eqref{eq:main-final}, we swap the sum and the product, up to an error term of $O\left(\frac{\log\log(3N)}{\log Q_{k,N}}\right)$:
\begin{align}\label{bigsum}\nonumber
\sum_{\ell=1}^t\sum_{n_1+\cdots+n_\ell=t}&\sum_{\substack{0\leq m_j\leq n_j\\ m_j\equiv n_j(2) \\ \forall1\leq j\leq\ell}}\frac{t!}{\ell! \ n_1!\cdots n_\ell!}\left(\prod_{j=1}^\ell c_{n_j,m_j}\right)\\
&\times\prod_{j=1}^\ell\left(\sum_{q\nmid N}\widehat{\phi}\left(\frac{\log q}{\log Q_{k,N}}\right)^{n_j}\frac{\chi(q)^{m_j}\log^{n_j} q}{q^{(n_j+m_j)/2}\log^{n_j} Q_{k,N}} \ S_r\left(q^{m_j}\right)\right).
\end{align}
We justify this step in detail in Appendix \ref{app:a}. 
We thereby reduce the problem to analyzing 
\begin{equation}\label{littlesum}
\mathcal{S}(n_j,m_j,\phi,\chi) \ := \ \sum_{q\nmid N}\widehat{\phi}\left(\frac{\log q}{\log Q_{k,N}}\right)^{n_j}\frac{\chi(q)^{m_j}\log^{n_j} q}{q^{(n_j+m_j)/2}\log^{n_j} Q_{k,N}} \ S_r\left(q^{m_j}\right)
\end{equation} 
for specific cases of $n_j$ and $m_j$. We exhibit this analysis separately for non-trivial and trivial $\chi$, in Lemmas \ref{lemma:Nontrivial-Analysis} and \ref{lemma:Trivial-Analysis} respectively.

\begin{lemma}\label{lemma:Nontrivial-Analysis}
Let $\chi$ be a non-trivial character of modulus $D\geq1$, and $\phi$ be a test function with compactly supported Fourier transform $\widehat{\phi}$. Let $n_j$ be a positive integer and $m_j$ be a non-negative integer such that $m_j\leq n$ and $m_j\equiv n_j\pmod{2}$. Then 
\begin{align}
\nonumber
\mathcal{S}(n_j,m_j,\phi,\chi) \ &= \ \sum_{q\nmid N}\widehat{\phi}\left(\frac{\log q}{\log Q_{k,N}}\right)^{n_j}\frac{\chi(q)^{m_j}\log^{n_j} q}{q^{(n_j+m_j)/2}\log^{n_j} Q_{k,N}} \ S_r\left(q^{m_j}\right) \\
&= \ 
\begin{cases}
\frac{\sigma_\phi^2}{4}+O\left(\frac{\log\log(3N)}{\log Q_{k,N}}\right) & (m_j,n_j)=(0,2), \\
O\left(\frac{\log\log(3N)}{\log Q_{k,N}}\right) & (m_j,n_j)=(1,1), \\
O\left(\frac{1}{\log^2Q_{k,N}}\right) & \text{otherwise}.
\end{cases}
\end{align}.
\end{lemma}
\begin{proof}
First, we consider $(m_j,n_j)=(0,2)$. In this case, we apply Abel's summation formula in tandem with the prime number theorem, recalling that $\phi$ is even and $\widehat{\phi}$ has compact support. Therefore, the sum under consideration equals
\begin{align}
\sum_{\substack{q\text{ prime}; \\ q\nmid N}}\widehat{\phi}\left(\frac{\log q}{\log Q_{k,N}}\right)^2\frac{\log^2q}{q\log^2Q_{k,N}} \ = \ \frac{\sigma_\phi^2}{4}+O\left(\frac{\log\log(3N)}{\log Q_{k,N}}\right).
\end{align}
This equation is also given in \cite[(3.4)]{HM07}. 

Second, we consider $(m_j,n_j)=(1,1)$. We follow the same method from the previous case and appeal to Dirichlet's theorem on primes in arithmetic progressions (the value of $\chi$ is $\pm1$ on exactly half the primes, up to lower order terms). This gives
\begin{align}
\nonumber
&\sum_{\substack{q\text{ prime}; \\ q\nmid N}}\widehat{\phi}\left(\frac{\log q}{\log Q_{k,N}}\right)\frac{\chi(q)\log q}{q\log Q_{k,N}} \ S_r(q) \\
\nonumber
= & \ \sum_{q\text{ prime}}\widehat{\phi}\left(\frac{\log q}{\log Q_{k,N}}\right)\frac{\chi(q)\log q}{q\log Q_{k,N}}+O\left(\frac{\log\log(3N)}{\log Q_{k,N}}\right) \\
\nonumber
= & \ \sum_{q:\chi(q)=1}\widehat{\phi}\left(\frac{\log q}{\log Q_{k,N}}\right)\frac{\log q}{q\log Q_{k,N}}-\sum_{q:\chi(q)=-1}\widehat{\phi}\left(\frac{\log q}{\log Q_{k,N}}\right)\frac{\log q}{q\log Q_{k,N}}+O\left(\frac{\log\log(3N)}{\log Q_{k,N}}\right) \\
= & \ O\left(\frac{\log\log(3N)}{\log Q_{k,N}}\right);
\end{align}
this equation is also given in \cite[165]{KR18}.

Finally, we consider $(m_j,n_j)\neq(0,2),(1,1)$. Since $m_j$ and $n_j$ are integers with the same parity, we know that $m_j+n_j\geq4$ in this case, and the sum on the right-hand side of \eqref{eq:0.4} converges: 
\begin{align}\label{eq:0.4}
\nonumber
\sum_{\substack{q\text{ prime}; \\ q\nmid N}}\widehat{\phi}\left(\frac{\log q}{\log Q_{k,N}}\right)^{n_j}\frac{\chi(q)^{m_j}\log^{n_j} q}{q^{(m_j+n_j)/2}\log^{n_j} Q_{k,N}} \ S_r\left(q^{m_j}\right) \ \ll_r & \ \frac{1}{\log^{n_j} Q_{k,N}}\sum_{q\text{ prime}}\frac{\log^{n_j} q}{q^{(m_j+n_j)/2}}\\
= & \ O\left(\frac{1}{\log^2 Q_{k,N}}\right).
\end{align}
This completes the analysis for $\chi$ non-trivial. 
\end{proof}

\begin{lemma}\label{lemma:Trivial-Analysis}
Let $\chi_0$ be the trivial character of modulus $D\geq1$, and $\phi$ be a test function with compactly supported Fourier transform $\widehat{\phi}$. Let $n_j$ be a positive integer and $m_j$ be a nonnegative integer such that $m_j\leq n_j$ and $m_j\equiv n_j\pmod{2}$. Then  
\begin{align}
\nonumber
\mathcal{S}(n_j,m_j,\phi,\chi_0) \ &= \ \sum_{q\nmid N}\widehat{\phi}\left(\frac{\log q}{\log Q_{k,N}}\right)^{n_j}\frac{\chi_0(q)^{m_j}\log^{n_j} q}{q^{(n_j+m_j)/2}\log^{n_j} Q_{k,N}} \ S_r\left(q^{m_j}\right) \\ 
&= \ \begin{cases}
\frac{\sigma_\phi^2}{4}+O\left(\frac{\log\log(3N)}{\log Q_{k,N}}\right) & (m_j,n_j)=(0,2), \\
\frac{\phi(0)}{2}+O\left(\frac{\log\log(3N)}{\log Q_{k,N}}\right) & (m_j,n_j)=(1,1), \\
O\left(\frac{1}{\log^2Q_{k,N}}\right) & \text{otherwise}.
\end{cases}
\end{align}
\end{lemma}
\begin{proof}
The analysis for $(m_j,n_j)\neq(1,1)$ is identical to Lemma \ref{lemma:Nontrivial-Analysis}, so we need only tend to $(m_j,n_j)=(1,1)$. We compute the corresponding sum using the methods from the previous lemma:
\begin{align}
\nonumber
&\sum_{\substack{q\text{ prime}; \\ q\nmid N}}\widehat{\phi}\left(\frac{\log q}{\log Q_{k,N}}\right)\frac{\chi_0(q)\log q}{q\log Q_{k,N}} \ S_r(q)\\ 
\nonumber
= \ & \sum_{q\text{ prime}}\widehat{\phi}\left(\frac{\log q}{\log Q_{k,N}}\right)\frac{\log q}{q\log Q_{k,N}}+O\left(\frac{\log\log(3N)}{\log Q_{k,N}}\right) \\
= \ & \frac{\phi(0)}{2} + O\left(\frac{\log\log(3N)}{\log Q_{k,N}}\right).
\end{align}
This equation is also given in \cite[164]{KR18}.
\end{proof}
With these asymptotic formulas for $\mathcal{S}(n_j,m_j,\phi,\chi)$, we analyze the average of interest \eqref{eq:3'}, which we showed to be equal to \eqref{bigsum} up to $O(\frac{\log\log(3N)}{\log Q_{k,N}})$, i.e.,
\begin{align}\nonumber
&\mathcal{A}^{w_{\chi,r}}_{\mathcal{F}_k(N)'}\left[\left(\sum\nolimits_{p\nmid N}\lambda_\cdot(p) \ \widehat{\phi}\left(\frac{\log p}{\log Q_{k,N}}\right)\frac{\log p}{\sqrt{p}\log Q_{k,N}}\right)^t\right] \\ 
\label{bigsum'}
&= \ \sum_{\ell=1}^t\sum_{n_1+\cdots+n_\ell=t}\sum_{\substack{0\leq m_j\leq n_j \\ m_j\equiv n_j(2) \\ \forall1\leq j\leq\ell}}\frac{t!}{\ell! \ n_1!\cdots n_\ell!}\left(\prod_{j=1}^\ell c_{n_j,m_j}\cdot \mathcal{S}(n_j,m_j,\phi,\chi)\right)\\
\label{bigsum'-error}
&+ \ O\left(\frac{\log\log(3N)}{\log Q_{k,N}}\right).
\end{align}
Using our analysis of \eqref{bigsum'}, we determine the $n$\textsuperscript{th} centered moment of the one-level density. We do these separately for non-trivial and trivial $\chi$ in Sections \ref{subsec:5.1} and \ref{subsec:5.2}, respectively.

%
%

\subsection{Non-Trivial Twisting Character}\label{subsec:5.1}
Every term of \eqref{bigsum'} corresponds to a choice of $1 \leq \ell \leq t$; a choice of $(n_1,\dots,n_{\ell})$ such that $\sum_{j=1}^\ell n_j = t$; and a choice of $(m_1,\dots,m_\ell)$ such that $m_j\equiv n_j\pmod{2}$ and $0 \leq m_j \leq n_j$ for all $j$. In view of Lemma \ref{lemma:Nontrivial-Analysis}, we split our consideration of these terms into two cases: terms for which $(m_j,n_j)\neq(0,2)$ for some $j$, and the term for which $(m_j,n_j)=(0,2)$ for all $j$. 

For terms of the former type, consider the $j$ for which $(m_j,n_j)\neq(0,2)$. For this $j$, Lemma \ref{lemma:Nontrivial-Analysis} yields
\begin{equation}
\mathcal{S}(n_j,m_j,\phi,\chi) \ = \ O\left(\frac{\log\log(3N)}{\log Q_{k,N}}\right);
\end{equation}
and for all the other $j$,
\begin{align}
\mathcal{S}(n_j,m_j,\phi,\chi) \ = \ O(1).
\end{align}
Multiplying the order of magnitude over all $j$, each of these terms contributes
\begin{equation}\label{Nontrivial-Error}
\frac{t!}{\ell! \ n_1!\cdots n_\ell!}\left(\prod_{j=1}^\ell c_{n_j,m_j}\cdot \mathcal{S}(n_j,m_j,\phi,\chi)\right) \ = \ O\left(\frac{\log\log(3N)}{\log Q_{k,N}}\right).
\end{equation}
Since the number of terms of this type depends only on $t$, the total contribution is also $O\left(\frac{\log\log(3N)}{\log Q_{k,N}}\right)$, the implied constant depending on $t$, $r$, and $D$. 

As for the latter case, we consider the term for which $(m_j,n_j)=(0,2)$ for all $j$. This arises only when $t$ is even and $\ell=t/2$ because $\sum_{j=1}^\ell n_j=t$. It follows from \eqref{eq:hecke-comb-recursive} that $c_{n_j,m_j}=1$ for all $j$; and Lemma \ref{lemma:Nontrivial-Analysis} gives 
\begin{equation}
\mathcal{S}(n_j,m_j,\phi,\chi) \ = \ \frac{\sigma_\phi^2}{4}+O\left(\frac{\log\log(3N)}{\log Q_{k,N}}\right)
\end{equation}
for all $j$.
Therefore, the contribution of this term to \eqref{bigsum'} is
\begin{equation}\label{Nontrivial-Main}
\frac{t!}{\left(t/2\right)! \ 2^{\ell/2}}\cdot\left(\frac{\sigma_\phi^2}{4}\right)^{t/2}+O\left(\frac{\log\log(3N)}{\log Q_{k,N}}\right) \ = \ (t-1)!! \ \left(\frac{\sigma_\phi^2}{4}\right)^{t/2}+O\left(\frac{\log\log(3N)}{\log Q_{k,N}}\right).
\end{equation}

Adding the contributions of all these terms \eqref{Nontrivial-Error} and \eqref{Nontrivial-Main}, we have that 
\begin{align}
\nonumber
&\mathcal{A}^{w_{\chi,r}}_{\mathcal{F}_k(N)'}\left[\left(\sum\nolimits_{p\nmid N}\lambda_\cdot(p) \ \widehat{\phi}\left(\frac{\log p}{\log Q_{k,N}}\right)\frac{\log p}{\sqrt{p}\log Q_{k,N}}\right)^t\right] \\
\label{eq:Moment-Nontrivial}
&= \ (t-1)!! \ \left(\frac{\sigma_\phi^2}{4}\right)^{t/2}+O\left(\frac{\log\log(3N)}{\log Q_{k,N}}\right)
\end{align}
Finally, taking $t = n$ in \eqref{eq:Moment-Nontrivial}, we deduce that the $n$\textsuperscript{th} centered moment of the $w_{\chi_0,r}$-weighted one-level density of $\mathcal{F}_{k}(N)'$, originally formulated in \eqref{eq:1}, is
\begin{align}
& \mathcal{A}^{w_{\chi,r}}_{\mathcal{F}_k(N)'}\left[\left(-2\sum_{p\nmid N}\lambda_\cdot(p)\widehat{\phi}\left(\frac{\log p}{\log Q_{k,N}}\right)\frac{\log p}{\sqrt{p}\log Q_{k,N}}+O\left(\frac{\log\log(kN)}{\log Q_{k,N}}\right)\right)^n\right] \\
& = \  (-2)^n \mathcal{A}^{w_{\chi,r}}_{\mathcal{F}_k(N)'}\left[\left(\sum_{p\nmid N}\lambda_\cdot(p)\widehat{\phi}\left(\frac{\log p}{\log Q_{k,N}}\right)\frac{\log p}{\sqrt{p}\log Q_{k,N}}\right)^n\right] + O\left(\frac{\log\log(kN)}{\log Q_{k,N}}\right) \\
& = \ (-2)^n \ (n-1)!! \ \left(\frac{\sigma_\phi^2}{4}\right)^{n/2}+O\left(\frac{\log\log(kN)}{\log Q_{k,N}}\right) \\
& = \ (n-1)!! \ \sigma_\phi^n+O\left(\frac{\log\log(kN)}{\log Q_{k,N}}\right),
\end{align} 
proving Theorem \ref{thm:main-thm} in the case that $\chi$ is non-trivial. 

\subsection{Trivial Twisting Character}\label{subsec:5.2}
Although the analysis for trivial $\chi = \chi_0$, the combinatorial arguments made to synthesize the various term-wise contributions in this case are more unassuming. 

In view of Lemma \ref{lemma:Trivial-Analysis}, we split our consideration of the terms of \eqref{bigsum'} into two cases: terms for which $m_j+n_j\geq3$ for some $j$, and terms for which $m_j+n_j\leq2$ for all $j$. 

For terms of the the former type, consider the $j$ for which $n_j+m_j\geq 3$. For this $j$, Lemma \ref{lemma:Trivial-Analysis} yields 
\begin{equation}
    \mathcal{S}(n_j,m_j,\phi,\chi_0) \ = \ O\left(\frac{1}{\log^2 Q_{k,N}}\right);
\end{equation}
and for all the other $j$, 
\begin{align}
\mathcal{S}(n_j,m_j,\phi,\chi_0) \ = \ O(1).
\end{align}
Multiplying the order of magnitude over all $j$, we find that each of these terms contributes
\begin{equation}\label{Trivial-Error}
\frac{t!}{\ell! \ n_1!\cdots n_\ell!}\left(\prod_{j=1}^\ell c_{n_j,m_j}\cdot \mathcal{S}(n_j,m_j,\phi,\chi_0)\right) \ = \ O\left(\frac{\log\log(3N)}{\log Q_{k,N}}\right).
\end{equation}

Next, we consider the terms for which $m_j+n_j\leq2$ for all $j$. Since $m_j$ and $n_j$ have the same parity, the set of possibilities for $(m_j,n_j)$ is $\{(0,2),(1,1)\}$; specializing (\ref{eq:hecke-comb-recursive}) for these possibilities gives $c_{n_j,m_j}=1$ for all $j$. Based on Lemma \ref{lemma:Trivial-Analysis}, if $n_j=1$ it contributes $\phi(0)/2$ to the main term; if $n_j=2$ it contributes $\sigma_\phi^2/4$ to the main term. If the number of $j$'s with $n_j=2$ is $0\leq s\leq\lfloor t/2\rfloor$, then the number of $j$'s with $n_j=1$ is $t-2s$ because $\sum_j n_j=t$; this case occurs $\binom{t-s}{s}$ times. Tallying the contribution in all these cases, we get 
\begin{align}\label{Trivial-Main}
\nonumber
&\frac{t!}{\ell! \ n_1!\cdots n_\ell!}\left(\prod_{j=1}^\ell c_{n_j,m_j}\cdot \mathcal{S}(n_j,m_j,\phi,\chi_0)\right) \\
& = \ \sum_{s=0}^{\lfloor t/2\rfloor}\frac{t!}{2^s(t-s)!}\binom{t-s}{s}\left(\frac{\phi(0)}{2}\right)^{t-2s}\left(\frac{\sigma_\phi^2}{4}\right)^{s}.
\end{align} 

Overall, adding \eqref{Trivial-Error} and \eqref{Trivial-Main}, we have that 
\begin{align}
\nonumber
&\mathcal{A}^{w_{\chi_0,r}}_{\mathcal{F}_k(N)'}\left[\left(\sum\nolimits_{p\nmid N}\lambda_\cdot(p) \ \widehat{\phi}\left(\frac{\log p}{\log Q_{k,N}}\right)\frac{\log p}{\sqrt{p}\log Q_{k,N}}\right)^t\right] \\
\label{eq:Moment-Trivial}
&= \ \sum_{s=0}^{\lfloor t/2\rfloor}\frac{t!}{2^s(t-s)!}\binom{t-s}{s}\left(\frac{\phi(0)}{2}\right)^{t-2s}\left(\frac{\sigma_\phi^2}{4}\right)^{s}+O\left(\frac{\log\log(3N)}{\log Q_{k,N}}\right)
\end{align}
Using \eqref{eq:Moment-Trivial}, we now calculate the $n$\textsuperscript{th} centered moment of the $w_{\chi_0,r}$-weighted one-level density of $\mathcal{F}_{k}(N)'$, originally formulated in \eqref{eq:1}:
\begin{align}
\nonumber
&\mathcal{A}^{w_{\chi,r}}_{\mathcal{F}_k(N)'}\left[\left(\phi(0)-2\sum_{p\nmid N}\lambda_\cdot(p)\widehat{\phi}\left(\frac{\log p}{\log Q_{k,N}}\right)\frac{\log p}{\sqrt{p}\log Q_{k,N}}+O\left(\frac{\log\log(kN)}{\log Q_{k,N}}\right)\right)^n\right]\\
\nonumber
& = \ 
\sum_{t=0}^n\binom{n}{t}\phi(0)^{n-t}(-2)^t \ \mathcal{A}^{w_{\chi,r}}_{\mathcal{F}_k(N)'}\left[\left(\sum_{p\nmid N}\lambda_\cdot(p)\widehat{\phi}\left(\frac{\log p}{\log(k^2N)}\right)\frac{\log p}{\sqrt{p}\log Q_{k,N}}\right)^t\right] \\
\nonumber
& + \ O\left(\frac{\log\log(kN)}{\log Q_{k,N}}\right)\\
\nonumber
& = \ \sum_{t=0}^n\binom{n}{t}\phi(0)^{n-t}(-2)^t\sum_{s=0}^{\lfloor t/2\rfloor}\frac{t!}{2^s(t-s)!}\binom{t-s}{s}\left(\frac{\phi(0)}{2}\right)^{t-2s}\left(\frac{\sigma_\phi^2}{4}\right)^{s} \\
& + \ O\left(\frac{\log\log(kN)}{\log Q_{k,N}}\right).
\end{align}
We denote the main term above by
\begin{equation}
\mathfrak{C} \ := \ \sum_{t=0}^n\binom{n}{t}\phi(0)^{n-t}(-2)^t\sum_{s=0}^{\lfloor t/2\rfloor}\frac{t!}{2^s(t-s)!}\binom{t-s}{s}\left(\frac{\phi(0)}{2}\right)^{t-2s}\left(\frac{\sigma_\phi^2}{4}\right)^{s}.
\end{equation}
We unravel $\mathfrak{C}$ not with the generating series trick of Soshnikov \cite[Lemma 2]{Sos00}, which is characteristic in $n$-level calculations, but with an elementary combinatorial trick of our own.

Regrouping the factors in each term of $\mathfrak{C}$ so that the respective index is most prominent in its summation,
\begin{equation}
\mathfrak{C} \ = \ \phi(0)^n\sum_{t=0}^n\binom{n}{t}(-1)^t\sum_{s=0}^{\lfloor t/2\rfloor}\frac{t!}{(t-s)!}\binom{t-s}{s}\left(\frac{\sigma_\phi^2}{2\phi(0)^2}\right)^{s}.
\end{equation}
Seeing the various parts of the typical term in the inner summation, particularly $t!$, $(t-2s)!$, and the cancelable $(t-s)!$, we expect to rewrite it in the form $\binom{t}{2s}(\cdot)^{2s}$ to subsequently use the binomial theorem, the preferred tool for this type of problem:
\begin{align}
\frac{t!}{(t-s)!}\binom{t-s}{s}\left(\frac{\sigma_\phi^2}{2\phi(0)^2}\right)^{s} \ 
&= \ \binom{t}{2s}\frac{(2s)!}{s!}\left(\frac{\sigma_\phi^2}{2\phi(0)^2}\right)^{s}.
\end{align}
Having obtained the desired binomial coefficient, we attend to the other factors in the term:
\begin{align}
\frac{t!}{(t-s)!}\binom{t-s}{s}\left(\frac{\sigma_\phi^2}{2\phi(0)^2}\right)^{s} \ 
&= \ \binom{t}{2s}(2s-1)!!\left(\frac{\sigma_\phi^2}{\phi(0)^2}\right)^{s}.
\end{align}
To write the term in the desired form, we let $X$ be a Gaussian random variable with mean $0$ and variance $\sigma_\phi^2/\phi(0)^2$. The $m$\textsuperscript{th} centered moment of $X$ is $0$ if $m$ is odd, and $(m-1)!!(\sigma_\phi^2/\phi(0)^2)^{m/2}$ if $m$ is even. Therefore,
\begin{align}
\frac{t!}{(t-s)!}\binom{t-s}{s}\left(\frac{\sigma_\phi^2}{2\phi(0)^2}\right)^{s} 
&= \ \binom{t}{2s}\mathbb{E}\left[X^{2s}\right].
\end{align}

Having tailored the problem to the binomial theorem, we use it to full avail, twice:
\begin{align}
\nonumber
\mathfrak{C} \ &= \ \phi(0)^n\sum_{t=0}^n\binom{n}{t}(-1)^t\sum_{s=0}^{\lfloor t/2\rfloor}\frac{t!}{(t-s)!}\binom{t-s}{s}\left(\frac{\sigma_\phi^2}{2\phi(0)^2}\right)^{s}\\
\nonumber
&= \ \phi(0)^n\sum_{t=0}^n\binom{n}{t}(-1)^t\sum_{s=0}^{\lfloor t/2\rfloor}\binom{t}{2s}\mathbb{E}\left[X^{2s}\right]\\
\nonumber
&= \ \phi(0)^n\sum_{t=0}^n\binom{n}{t}(-1)^t\sum_{r=0}^{t}\binom{t}{r}\mathbb{E}\left[X^{r}\right]\\
\nonumber
&= \ \phi(0)^n\sum_{t=0}^n\binom{n}{t}(-1)^t \ \mathbb{E}\left[(1+X)^t\right]\\
\nonumber
&= \ \phi(0)^n \ \mathbb{E}\left[(-X)^n\right]\\
&= \ \begin{cases}
0 & n \text{ odd}, \\
\phi(0)^{n}(n-1)!!\left(\frac{\sigma_\phi^2}{\phi(0)^2}\right)^{n/2}=(n-1)!! \ \sigma_\phi^{n}& n\text{ even},
\end{cases}
\end{align}
which is precisely the $n$\textsuperscript{th} centered moment of a Gaussian distribution with variance $\sigma_\phi^2$, completing the proof of Theorem \ref{thm:main-thm} in the case of $\chi=\chi_0$ trivial too. 

\appendix

%
%

\section{Switching Sums and Products}\label{app:a}
Switching the following sum and product is what allows us to carry out the combinatorial analysis in Section \ref{sec:5}. We present this crucial, albeit cumbersome, step here. For clarity of the mainly combinatorial aspects of this step, we adopt the notation: 
\begin{align}
T(q,n,m,\chi)\ := \ \widehat{\phi}\left(\frac{\log q}{\log Q_{k,N}}\right)^{n}\frac{\chi(q)^{m}\log^{n} q}{q^{(n+m)/2}\log^{n} Q_{k,N}} \ S_r\left(q^{m}\right).
\end{align} Furthermore, we abbreviate ``not necessarily distinct'' as ``n.n.d.''. 
\begin{claim} 
For all $1\leq j\leq \ell$, let $m_j$ and $n_j$ be non-negative integers such that $m_j\leq n_j$ and $m_j\equiv n_j\pmod{2}$. Then 
\begin{align}\nonumber
\sum_{\substack{(q_1,\dots,q_\ell) \\ \text{distinct;} \\ q_j\nmid N}}\prod_{j=1}^\ell T(q_j,n_j,m_j,\chi) \ = \ \prod_{j=1}^\ell\left(\sum_{q_j\nmid N}T(q_j,n_j,m_j,\chi)\right) + O\left(\frac{1}{\log^2 Q_{k,N}}\right).
\end{align}
\end{claim}
\begin{proof}
We would like to prove that
\begin{align}
\nonumber& \prod_{j=1}^\ell\left(\sum_{q_j\nmid N}T(q_j,n_j,m_j,\chi)\right) - \sum_{\substack{(q_1,\dots,q_\ell) \\ \text{distinct;} \\ q_j\nmid N}}\prod_{j=1}^\ell T(q_j,n_j,m_j,\chi) \\
\nonumber= \ & \sum_{\substack{(q_1,\dots,q_\ell) \\ \text{n.n.d.}; \\ q_j\nmid N}}\prod_{j=1}^\ell T(q_j,n_j,m_j,\chi) - \sum_{\substack{(q_1,\dots,q_\ell) \\ \text{distinct;} \\ q_j\nmid N}}\prod_{j=1}^\ell T(q_j,n_j,m_j,\chi) \\
\label{eq: Error 0.1}
= \ & \sum_{\substack{(q_1,\dots,q_\ell) \\ \text{n.n.d.}; \\ q_j\nmid N; \\ \exists j_1<j_2: \ q_{j_1}=q_{j_2}}}\left(T(q_{j_1},n_{j_1},m_{j_1},\chi)\cdot T(q_{j_2},n_{j_2},m_{j_2},\chi)\cdot\prod_{\substack{1\leq j\leq\ell; \\ j\neq j_1, j_2}} T(q_j,n_j,m_j,\chi)\right)
\end{align} 
has order at most $\log^{-2}Q_{k,N}$. By the principle of inclusion-exclusion, we know that
\begin{align}
\nonumber& \sum_{\substack{(q_1,\dots,q_\ell) \\ \text{n.n.d}; \\ q_j\nmid N; \\ \exists j_1<j_2: \ q_{j_1}=q_{j_2}}}\left(T(q_{j_1},n_{j_1},m_{j_1},\chi)\cdot T(q_{j_2},n_{j_2},m_{j_2},\chi)\cdot\prod_{\substack{1\leq j\leq\ell; \\ j\neq j_1, j_2}} T(q_j,n_j,m_j,\chi)\right)\\
\label{eq:Inclusion-Exclusion}
\leq \ & \sum_{\substack{1\leq j_2\leq \ell;\\1\leq j_1<j_2}}\sum_{\substack{(q_1,\dots,q_{j_2-1},\\ q_{j_2+1},\dots,q_\ell) \\ \text{n.n.d.}; \\ q_j\nmid N}}\left(T(q_{j_1},n_{j_1},m_{j_1},\chi)\cdot T(q_{j_1},n_{j_2},m_{j_2},\chi)\cdot\prod_{\substack{1\leq j\leq\ell, \\ j\neq j_1, j_2}} T(q_j,n_j,m_j,\chi)\right).
\end{align} 
In fact, for $\ell\geq3$, the above inequality is strict because we double-count several terms. For one, consider $q_1=q_2=q_3$ and $q_3,\ldots,q_\ell$ distinct; the corresponding term $\prod_{j}T(q_j,n_j,m_j,\chi)$ is counted three times in (\ref{eq:Inclusion-Exclusion}) when $(j_1,j_2)=(1,2),(1,3),$ and $(2,3)$. There are many other examples of double-counted terms. 

Since $q_1,\dots,q_{j_2-1},q_{j_2+1},\dots,q_\ell$ are n.n.d.\ in (\ref{eq:Inclusion-Exclusion}), we can freely swap the sum and product:
\begin{align}
& \sum_{\substack{1\leq j_2\leq \ell;\\1\leq j_1<j_2}}\sum_{\substack{(q_1,\dots,q_{j_2-1},\\ q_{j_2+1},\dots,q_\ell) \\ \text{n.n.d.}; \\ q_j\nmid N}}\left(T(q_{j_1},n_{j_1},m_{j_1},\chi)\cdot T(q_{j_1},n_{j_2},m_{j_2})\cdot\prod_{\substack{1\leq j\leq\ell, \\ j\neq j_1, j_2}} T(q_j,n_j,m_j,\chi)\right) \\
\label{eq:PIE-Analysis}
= \ & \sum_{\substack{1\leq j_2\leq \ell;\\1\leq j_1<j_2}}\left[\left(\sum_{q_{j_1}\nmid N}T(q_{j_1},n_{j_1},m_{j_1},\chi)\cdot T(q_{j_1},n_{j_2},m_{j_2},\chi)\right)\cdot\prod_{\substack{1\leq j\leq\ell \\ j\neq j_1,j_2}}\left(\sum_{q_j\nmid N}T(q_j,n_j,m_j,\chi)\right)\right].
\end{align}

From Lemmas \ref{lemma:Nontrivial-Analysis} and \ref{lemma:Trivial-Analysis}, we know that $\sum T(q_j,n_j,m_j,\chi)=O(1)$ for all $j\neq j_1,j_2$. Furthermore, noting that $m_{j_i}+n_{j_i}\geq2$ for both $i=1,2$, we get that
\begin{align}
\nonumber\sum_{q_{j_1}\nmid N}T(q_{j_1},n_{j_1},m_{j_1},\chi)\cdot T(q_{j_1},n_{j_2},m_{j_2},\chi) \ \ll_r & \ \sum_{q_{j_1}\nmid N}\frac{\log^{n_{j_1}+n_{j_2}} q}{q^{(m_{j_1}+n_{j_1}+m_{j_2}+n_{j_2})/2}\log^{n_{j_1}+n_{j_2}} Q_{k,N}}\\
= & \ O\left(\frac{1}{\log^2 Q_{k,N}}\right).
\end{align}
Absorbing the outer sum into the constant (as $\ell\leq t\leq n$) and multiplying the orders of magnitude of all these factors, we find that the total order of magnitude of (\ref{eq:PIE-Analysis}) and hence is at most $\log^{-2}Q_{k,N}$, as claimed. 
\end{proof}

\section*{References}

\printbibliography[heading=none]

\end{document}